\documentclass[a4paper, 12pt]{amsart}

\usepackage{amssymb, amsmath, amsthm}
\usepackage{mathrsfs}
\usepackage{amscd}
\usepackage[all]{xy}
\usepackage{color}

\title
[Fundamental class]
{Orientability and fundamental classes of Alexandrov spaces with applications}

\author{Ayato Mitsuishi}
\date{\today}

\theoremstyle{plain}
\newtheorem{theorem}{Theorem}[section]
\newtheorem{lemma}[theorem]{Lemma}
\newtheorem{corollary}[theorem]{Corollary}
\newtheorem{proposition}[theorem]{Proposition}
\newtheorem{definition}[theorem]{Definition}
\newtheorem{remark}[theorem]{Remark}
\newtheorem{example}[theorem]{Example}
\newtheorem{problem}[theorem]{Problem}
\newtheorem{conjecture}[theorem]{Conjecture}
\newtheorem{sublemma}[theorem]{Sublemma}
\newtheorem{claim}[theorem]{Claim}

\newtheorem{nonexample}[theorem]{Non-example}

\newtheorem*{conclusion}{Conclusion}

\newcommand{\fillrad}[1]{\mathrm{Fill.rad}(#1)}

\newcommand{\supp}{\mathrm{supp}}
\newcommand{\N}{\mathbf N}
\newcommand{\M}{\mathbf M}
\newcommand{\I}{\mathbf I}
\newcommand{\D}{\mathcal D}
\newcommand{\jump}[1]{\ensuremath{[\![#1]\!]}}
\renewcommand{\top}{\mathrm{top}}
\newcommand{\Lip}{\mathrm{Lip}}
\newcommand{\cpt}{\mathrm{c}}
\newcommand{\mass}{\mathrm{mass}}
\newcommand{\Hau}{\mathcal H}
\newcommand{\ori}{\mathrm{ori}}

\def\XXint#1#2#3{{\setbox0=\hbox{$#1{#2#3}{\int}$}
\vcenter{\hbox{$#2#3$}}\kern-.5\wd0}}

\makeatletter

 \@addtoreset{equation}{section}
\makeatother

\begin{document}
\maketitle

\begin{abstract}
In the present paper, we consider several valid notions of orientability of Alexandov spaces and prove that all such conditions are equivalent. 
Further, we give topological and geometric applications of the orientability. 
In particular, a Poincar\'e-type duality theorem is proved.
As a corollary to the duality theorem, we also prove that if a closed Alexandrov space admits a positive 
curvature bound in a synthetic sense, then its codimension one homology vanishes. 
Further, we obtain a filling radius inequality for closed orientable Alexandrov spaces. 
\end{abstract}

\section{Introduction} \label{sec:intro}
Alexandrov spaces naturally appear in the collapsing theory of Riemannian manifolds, 
which are by the definition, metric spaces satisfying a condition of having a lower sectional curvature bound in a synthetic sense. 
In the point of view of Riemannian geometry, Alexandrov spaces are considered as natural generalized objects of Riemannian manifolds.
It has been important to study their geometry and topology thoroughly. 
The definition of Alexandrov spaces will be recalled in Section \ref{sec:Alex}.

Topological manifolds are by the definition, Hausdorff and locally Euclidean spaces, and the orientability is a fundamental global topological property for manifolds. 
As generalizations of manifolds, more locally complicated spaces have been researched. 
For instance, homology manifolds or psuedo-manifolds are famous generalized spaces, and for those spaces, the notion of orientability is well-posed.
Such spaces are assumed to be triangulable, in almost cases.
On the other hands, it is not known whether Alexandrov spaces are triangulable or not.
Further, in general, Alexandrov spaces are not homology manifolds. 
For instance, the topological cones over $\mathbb RP^2$ and $\mathbb CP^2$ admit metrics of Alexandrov spaces, but are not homology manifolds (for given homology theory with suitable coefficients).
Let us list typical examples of Alexandrov spaces, here.

\begin{example}[\cite{BGP}] \upshape \label{ex:Alex}
\mbox{}
\begin{itemize}
\item Any connected complete Riemannian manifold (moreover, any connected complete Riemannian orbifold) is an Alexandrov space. 
\item For a connected complete Riemannian manifold $M$ of sectional curvature $\ge \kappa$ with an isometric action of a group $G$, the completion of the orbit space $M/G$ is an Alexandrov space of curvature $\ge \kappa$, where $\kappa \in \mathbb R$.
\item If $M$ and $N$ are connected complete Riemannian manifolds of sectional curvature $\ge 1$, then the join $M \ast N$ of them is an Alexandrov space of curvature $\ge 1$, the suspension $S^0 \ast M$ of $M$ is an Alexandrov space of curvature $\ge 1$ and the cone $c(M)$ over $M$ is an Alexandrov space of curvature $\ge 0$. 
\item The product of two Alexandrov spaces is an Alexandrov space.
\item The Gromov-Hausdorff limit of connected complete Riemannian manifolds of sectional curvature $\ge \kappa$ is an Alexandrov space of curvature $\ge \kappa$, where $\kappa \in \mathbb R$.
\end{itemize}
Here, the join, the suspension, the cone and the product are equipped with natural distance functions. 
Further, the above statements are true if we replace ``Riemannian manifold(s)'' with ``Alexandrov space(s)''.
\end{example}
By the above examples, Alexandrov spaces admit complicated topologies compared to smooth manifolds.
As a complemental remark, 
it is not known whether any topological manifold admit a metric of an Alexandrov space (cf. \cite{GZ}).

Under these circumstances, 
the notions of orientability for Alexandrov spaces were considered in several ways.
Let us recall a brief history of notions of orientations on Alexandrov spaces.
The notions of orientability were first formulated in different ways by Yamaguchi (\cite{Y}) and Petrunin (\cite{HS}), independently, that were topological properties. 
However, in \cite{Pet}, an explicit formulation was not given. 
Recently, Harvey and Searle (\cite{HS}) gave the notion of orientability which is regarded as a regorus definition of the orientability considered in \cite{Pet}.
Further, they proved that the notions of orientability given in the manuscripts \cite{Y}, \cite{Pet} and \cite{HS} are equivalent. 
On the other hands, it was known that Alexandrov spaces have the canonical smooth atlas in a weak sense (\cite{OS}). 
So, one can consider the notion of orientability in an analytic sense, that is, in the terminologies of coordiante transformations with respect to the weak smooth atlas. 
Such a condition was implicitely considered by Kuwae, Machigashira and Shioya (\cite{KMS}). 
We will discuss about orientations in this sense, in Appendix \ref{sec:C1-ori}.
In the present paper, we will consider other valid (topological and geometric) notions of orientability for Alexandrov spaces and prove that they are equivalent (Theorem \ref{thm:main thm}).
After establishing the meaning of the orientability by proving Theorem \ref{thm:main thm}, we give several topological and geometric applications of the orientability to Alexandrov spaces.

In the present paper, $H^\ast(-;G)$ and $H_\ast(-;G)$ denote the singular (co)homology with coefficients in an abelian group $G$. 
Further, $H_\cpt^\ast(-;G)$ denotes the singular cohomology of compact support. 
In the following statements, it is important what the coefficients of the (co)homologies are.

Main topological results are the following.
\begin{theorem}[Poincar\'e-type duality theorem] \label{thm:PD1}
Let $M$ be a non-empty open subset of an Alexandrov $n$-space having no boundary points. 
If it is orientable, then there is a natural morphism 
\[
D_M : H_\cpt^k(M;\mathbb Q) \to H_{n-k}(M;\mathbb Q)
\] 
for $0 \le k \le n$ such that it is an isomorphism when $k=n$ and $k=n-1$.
Here, the naturality means that if $N$ is another non-empty open subset of an Alexandrov $n$-space having no boundary points which is orientable and if $f : M \to N$ is a continuous 
embedding preserving orientations, then we have $D_N \circ f_\ast = f_\ast \circ D_M$.
\end{theorem}

An Alexandrov space is said to be closed if it is compact and without boundary.
\begin{corollary}
Let $M$ be a closed orientable Alexandrov $n$-space.
Then, $D_M : H^k(M;\mathbb Q) \to H_{n-k}(M;\mathbb Q)$ is isomorphic when $k=0,1,n-1,n$.
\end{corollary}


The suspension $X = S^0 \ast \mathbb CP^\ell$ over the complex projective space $\mathbb CP^\ell$ is an Alexandrov spaces (Example \ref{ex:Alex}). 
Then, $X$ is orientable in our sense. 
Remark that 
$H^k(X;\mathbb Q)$ and $H_{2 \ell -k+1}(X;\mathbb Q)$ never be isomorphic for middle degree $k$, that is, for $k$ with $2 \le k \le 2 \ell -1$.
See also Remark \ref{rem:middle degree}.

\begin{corollary} \label{cor:positive}
If $\Sigma$ is an $n$-dimensional Alexandrov space of curvature $\ge \kappa$, where $n \ge 2$ and $\kappa > 0$, then the codimension one rational homology $H_{n-1}(\Sigma;\mathbb Q)$ vanishes. 
In addition, if $\Sigma$ is orientable, then $H_{n-1}(\Sigma;\mathbb Z) = 0$.
\end{corollary}

We obtain an obstruction to the topology of Alexandrov spaces:
\begin{theorem} \label{thm:Alex top}
Let $M$ be an $n$-dimensional Alexandrov space without boundary, where $n \ge 2$. 
Then, for each $x \in M$, we have, either, 
\[
H_n(M, M \setminus \{x\}; \mathbb Z) \cong \mathbb Z \text{ and } H_{n-1}(M,M \setminus \{x\}; \mathbb Z) \cong 0,
\]
or 
\[
H_n(M, M \setminus \{x\}; \mathbb Z) = 0 \text{ and } H_{n-1}(M,M \setminus \{x\}; \mathbb Z) \cong \mathbb Z_2.
\]

Further, when $M$ is the boundary of an $(n+1)$-dimensional Alexandrov space, the same statement as above holds.
\end{theorem}
\begin{proof}
Combine Corollaries \ref{cor:positive} and \ref{cor:non-ori}, and Theorems \ref{thm:main thm} (A), (B), \ref{thm:stab} and \ref{thm:positive boundary}.
\end{proof}

\subsection{Non-branching MCS-spaces without boundary and the orientability}
Before stating our results about orientations, we should recall Perelman's stability theorem (Theorem \ref{thm:stab}) about the underlying topologies of Alexandrov spaces. 
Perelman proved 
that any Alexandrov space is a locally cone-like space, that is, any point in the space has a neighborhood which is homeomorphic to the cone over a compact space (\cite{Per:Alex2}, \cite{Per:Morse}). 
Moreover, the compact space as a generator of the cone can be taken to be an Alexandrov space. 
The last assertion is called the stability theorem. 
See Thereom \ref{thm:stab}, for details.
His proof was based on a deep study of the geometric topology due to Siebenmann (\cite{Sie}). 
So, to use the theory of \cite{Sie}, he prepared a class of locally cone-like spaces, called {\it MCS-spaces} (Defnition \ref{def:MCS}). 
However, the class of MCS-spaces is too large to describe the topologies of Alexandrov spaces. 
For instance, a generic graph is an MCS-space, but does not admit a metric of an Alexandrov space.
Actually, a graph with a vertex $v$ of degree at least three equipped with a natual length metric has curvature $-\infty$ at $v$ in the sense of Alexandrov.

Taking into account the stability theorem, Harvey and Searle (\cite{HS}) pointed out that Alexandrov spaces actually belong to a more restricted class of certain MCS-spaces that were called {\it non-branching MCS-spaces}.
The graph appeared just before is not an non-branching MCS-space.
We will deal with a little bit more restricted class of non-branching MCS-spaces, as a generalization of the class of Alexandrov spaces. 
A space belonging to our class will be called an {\it NB-space with or without boundary} (Definitions \ref{def:NB} and \ref{def:NB with boundary}). 
For sake of simplicity, NB-spaces without boundary are called NB-spaces.
Here, our NB-spaces are nothing but non-branching MCS-spaces without boundary in the sense of \cite{HS}. 
On the other hands, the class of NB-spaces with boundary is strictly smaller than the class of non-branching MCS-spaces with boundary.
Since the definition of NB-spaces with boundary is somewhat complicated, we put off the discussion about those spaces to Appendix \ref{sec:NB with boundary}.

As will be mentioned later, at least two notions of orientability which are naturally inherited to NB-spaces from Alexandrov spaces (without boundary) were proved to be equivalent (\cite{HS}, see Lemma \ref{lem:HS}).
In that sense, the class of NB-spaces was known to be suitable to discuss their orientability.
One of purposes of the present paper is to show that the class of NB-spaces (with or without boundary) is well suited to define the notions of orientability. 
Until the end of the subsection, we are going to state results on NB-spaces and their orientability. 
Let us recall the definition of NB-spaces: 

\begin{definition}[\cite{HS}] \upshape \label{def:NB}
We say that a separable metrizable space $X$ is an $n$-dimensional {\it non-branching MCS-space without boundary} (in short, {\it NB-space}) if it is a topological one-manifold without boundary when $n=1$; or it satisfies that for any $x \in X$, there exist an open neighborhood $U$ of $x$ and a compact connected NB-space $\Sigma$ of dimension $n-1$ such that $(U,x)$ is homeomorphic to $(c(\Sigma), o)$ as pointed spaces when $n \ge 2$, where $c(A)$ denotes the topological open cone $c(A) = A \times [0, \infty) / A \times \{0\}$ over a space $A$ and $o$ is the apex of it. 
In this case, $U$ is called a {\it conical neighborhood at $x$} and $\Sigma$ is a {\it generator} of the conical neighborhood $U$. 
\end{definition}

Typical examples considered here are Alexandrov spaces without boundary. 
The cone over the union of two disjoint circles is not an NB-space, but it is an MCS-space. 
Manifolds with boundary are not NB-spaces. 
Other examples and non-examples of NB-spaces are written in Section \ref{sec:prelim}. 

Any NB-space $X$ of dimension $n$ is known to have an open dense subset $X_\top$ such that $X_\top$ itself is an $n$-manifold without boundary (Lemma \ref{lem:open dense mfd}). 
One of valid definitions of orientability for $X$ is the following: 
we temporarily say that $X$ is ``orientable'' if $X_\top$ is orientable as an $n$-manifold. 
As an argument written in \cite{Y} (cf.\! \cite{GP}), we immediately know that if $X$ is compact connected and ``orientable'', then $H^n(X;\mathbb Z) \cong \mathbb Z$. 
Here, the following natual and naive problem arises.
\begin{problem} \label{problem}
Is it true that if a compact connected $n$-dimensional NB-space $X$ is ``orientable'', then $H_n(X;\mathbb Z)$ is isomorphic to $\mathbb Z$? Is the converse statement true?
\end{problem}
In other words, the problem ask us whether an ``orientation'' determines and is determined by the existence of fundamental class.
A starting point of the present paper is to answer to the problem in affirmatively. 
More generally, we will prove Theorem \ref{thm:main thm} as follows. 
The precise meaning of phrases in each statement of Theorem \ref{thm:main thm}, will be defined in Section \ref{sec:prelim}. 
\begin{theorem} \label{thm:main thm}
Let $X$ be an $n$-dimensional connected NB-space. 
Then, we have 
\begin{itemize}
\item[(A)] The following conditions are equivalent: 
\begin{itemize}
\item[(a)] the manifold-part $X_\top$ of $X$ is orientable as an $n$-manifold; 
\item[(b)] $H^n_\cpt(X;\mathbb Z) \cong \mathbb Z$ and the canonical morphism $H_\cpt^n(X;\mathbb Z) \leftarrow H^n(X,X \setminus \{x\};\mathbb Z)$ is isomorphic for any $x \in X$; 
\item[(c)] $X$ is cohomologically orientable; 
\item[(d)] $X$ is homologically orientable.
\end{itemize}
Further, if $X$ is compact, then the following conditions $(\mathrm{e})$--$(\mathrm{h})$ are also equivalent to the above conditions $(\mathrm{a})$--$(\mathrm{d})$: 
\begin{itemize}
\item[(e)] $H_n(X;\mathbb Z) \cong \mathbb Z$; 
\item[(f)] $H_n(X;\mathbb Z) \neq 0$;
\item[(g)] $H^n(X;\mathbb Z) \cong \mathbb Z$.
\item[(h)] the rank of $H^n(X;\mathbb Z)$ is positive.
\end{itemize}
Moreover, if $H_n(X;\mathbb Z) \neq 0$, then $X$ is compact and the canonical morphism $H_n(X;\mathbb Z) \to H_n(X,X \setminus \{x\};\mathbb Z)$ is isomorphic for any $x \in X$.
\item[(B)]
If $X$ is compact and orientable in the sense that it satisfies one of the conditions in $(\mathrm{A})$, then $H_{n-1}(X;\mathbb Z)$ has no torsion. 
\item[(C)]
If $X$ is non-compact (which is possibly orientable or not), then $H_n(X;G) = 0$ for any abelian group $G$.
\end{itemize}
\end{theorem}

Theorem \ref{thm:main thm} (A) and (C) gives an affirmative answer to Problem \ref{problem} as follows.
\begin{corollary} \label{cor:answer}
Let $X$ be a connected NB-space. Then, the condition that $X$ is compact and $X_\top$ is orientable, is equivalent to 
$H_n(X;\mathbb Z) \neq 0$.
Moreover, if $X$ satisfies one of the equivalent conditions, then $H_n(X;\mathbb Z) \cong \mathbb Z$ and the canonical morphism $H_n(X;\mathbb Z) \to H_n(X,X \setminus \{x\}; \mathbb Z)$ is an isomorphism for each $x \in X$.
\end{corollary}

Due to Theorem \ref{thm:main thm} (A), we fix the notion of orientability for NB-spaces: 
\begin{definition} \upshape \label{def:ori}
We say that a connected NB-space is {\it orientable} if it satisfies one of the conditions (a)--(d) listed in Theorem \ref{thm:main thm} (A).
For a general NB-space, it is {\it orientable} if each component is orientable. 
\end{definition}

Let us explain what the conditions had been considered in order to define the orientability of Alexandrov spaces, 
by using the statement of Theorem \ref{thm:main thm}.
The condition (a) was employed as the orientability in \cite{Y}. 
Although an explicit formulation was not given, in \cite{Pet}, 
it might be considered that an Alexandrov space without boundary  
is orientable if it satisfies (b) (or (c)). 
Here, the conditions (b) and (c) are easily known to be equivalent. 
Harvey and Searle (\cite{HS}) considered the notion of orienability for connected NB-spaces as the condition (b) 
and proved that (a) and (b) are equivalent (Lemma \ref{lem:HS}). 
An essentially new condition to define the orientability is the condition (d). 
However, (d) is the same as the usual condition to define the orientability for topological manifolds. 
So, the form of the condition (d) itself is not new. 
A point of this paper is to clear mutual relations among the conditions.
We should notice that when the conditions (a) and (b) (and (c)) are admitted to be equivalent, (d) is stronger than them (See Section \ref{sec:proof of AB}).
To prove that (d) and (b) are equivalent, we need the statement (B) and (C). 

Further, the statements (A), (B) and (C) themselves in Theorem \ref{thm:main thm} are mutually related. 
Indeed, denoting by (A)$_n$, (B)$_n$ and (C)$_n$ the statements (A), (B) and (C) for $n$-dimensional connected NB-spaces, respectively, we will prove the following implications 
\begin{align*}
&(\mathrm{A})_{n-1}, (\mathrm{B})_{n-1} \text{ and } (\mathrm{C})_n \implies (\mathrm{A})_n; \\
&(\mathrm{A})_n \text{ and } (\mathrm{C})_n \implies (\mathrm{B})_n; \\
&(\mathrm{A})_{n-1} \text{ and } (\mathrm{B})_{n-1} \implies (\mathrm{C})_n. 
\end{align*}


The following are immediate consequences of Theorem \ref{thm:main thm}.
\begin{corollary}
If an NB-space $X$ is orientable, then any generator at any point is orientable.
\end{corollary}

\begin{theorem} \label{thm:compact}
For an $n$-dimensional compact connected orientable NB-space $X$, we have that $H_n(X;G) \cong G$ and $H_n(X;G) \to H_n(X, X \setminus \{x\};G)$ is isomorphic for every $x \in X$ and for every abelian group $G$.
\end{theorem}

\begin{corollary} \label{cor:dual}
If $X$ is an $n$-dimensional compact connected orientable NB-space, then the cap product to a generator $[X]$ of $H_n(X;\mathbb Z)$ gives an isomorphism $[X] \cap - : H^n(X;G) \to H_0(X;G)$, where $G$ is an abelian group.
\end{corollary}
In general, $[X] \cap - : H^{n-1}(X;G) \to H_1(X;G)$ is not an isomorphism (see Remark \ref{rem:middle degree}).

\subsection{Further applications to Alexandrov spaces}
Applications of the orientability to Alexandrov spaces were mentioned in the start of the introduction. 
Except them, we also exhibit topological and geometric applications to Alexandrov spaces. 

An {\it Alexandrov domain} denotes a connected open subset of an Alexandrov space. 
Note that this term is different from the one used in \cite{Pet} and \cite{HS}.
When a non-empty Alexandrov domain is compact, it is automatically the whole Alexandrov space.

\begin{corollary} \label{cor:bdry}
If an $n$-dimensional Alexandrov domain $X$ meets the boundary of the whole space, then $H_n(X;G)=0$ for any abelian group $G$.
\end{corollary}

\begin{theorem} \label{thm:double}
Let $M$ be an $n$-dimensional Alexandrov space with boundary. 
Then, the interior of $M$ is orientable if and only if the double $D(M)$ of $M$ is orientable.
\end{theorem}
Here, the double of an Alexandrov space with boundary $M$ is a space obtained by gluing two copies of $M$ along their boundaries in a natural way, that is known to become an Alexandrov space without boundary (\cite{Per:Alex2}).

\begin{theorem} \label{thm:ori to bdry}
If an Alexandrov domain $M$ with boundary is orientable, then its boundary $\partial M$ is orientable as an NB-space.
\end{theorem}


To state geometric applications to the orientability of Alexandrov spaces, we need to know the notion of metric currents introduced by Ambrosio and Kirchheim (\cite{AK}). 
The definition of currents will be recalled in Section \ref{sec:current}.

\begin{theorem} \label{thm:current}
Let $n \ge 1$ and $M$ an Alexandrov space of dimension $n$.
Then the following holds. 
\begin{itemize}
\item 
If $M$ is closed and orientable, then the group
\[
\{ T \in \I_n(M) \mid \partial T = 0 \}
\]
is isomorphic to $\mathbb Z$. 
Here, $\I_n(M)$ denotes the group of all integral $n$-currents in $M$. 
In particular, there is a non-trivial integral $n$-current in $M$. 
\item If $M$ is closed and non-orientable, then the group 
\[
\{ T \in \N_n(M) \mid \partial T = 0 \}
\]
is trivial.
Here, $\N_n(M)$ denotes the group of all normal currents in $M$.
\item If $X$ is a non-compact open coonected subset of $M$,
then the group 
\[
\{ T \in \N_n^\cpt(M) \mid \partial T=0 \}
\]
is trivial, where $\N_n^\cpt(M)$ is the subgroup of $\N_n(M)$ consisting of all currents with compact support.
\end{itemize}
\end{theorem}

\begin{theorem} \label{thm:fill rad Alex} 
Let $M$ be an $n$-dimensional closed orientable Alexandrov space, with $n \ge 1$. 
Then, denoting by $\fillrad{M}$ the filling radius of $M$ introduced in Section \ref{sec:current}, we have the followng universal inequality
\[
\fillrad{M} \le C(n) \Hau^n(M)^{1/n}. 
\]
Here, $\Hau^n$ denotes the $n$-dimensional Hausdorff measure and $C(n)$ is a positive explicit constant depending on $n$.
\end{theorem}
Theorem \ref{thm:fill rad Alex} is new even for Riemannian oribifolds. 
The inequality in the theorem is called a filling radius inequality. 
Originally, Gromov proved the filling radius inequality for Riemanian manifolds in \cite{Gr}, where he used the original filling radius defined by himself. 
Our filling radius is a counterpart of Gromov's filling radius in terms of metric currents, defined by Ambrosio and Katz (\cite{AKt}). 
Remark that the inequality in Theorem \ref{thm:fill rad Alex} is scale-invariant and curvature-free.
As a related topic, Yokota gave an estimate of the filling radius for Alexandrov spaces of positive curvature depending on the lower curvature bound and investigated its rigidity (\cite{Yo}).
There, the filling radius was used in the original sense (of real or $\mathbb Z_2$-coefficients). 

As generalizations of complete Riemanian manifolds, metric measure spaces satisfying Riemannian curvature-dimension condition (abbreviated by RCD-spaces) are introduced by Ambrosio, Gigli and Savar\'e (\cite{AGS}) (See also \cite{EKS}).
Due to \cite{Pet:AmLVS}, \cite{ZZ} and \cite{GKO}, it is known that Alexandrov spaces are RCD-spaces. 
The measured Gromov-Hausdorff limits of complete Riemannian manifolds with a uniform Ricci lower curvature bound, called Ricci-limit spaces, are typical examples of RCD-spaces, and are important in the point of view of the collapsing theory of Riemannian manifolds. 
Geometric-analytic properties of RCD or Ricci-limit spaces have been actively investigated.
On the other hands, it is hard to understand and investigate topological properties of such spaces.
Recently, Honda (\cite{Honda:ori}) also successed to introduce the notion of orientability on Ricci-limit spaces 
by using the terminologies of rectifiable atlases.
That is, his definition of orientations is based on an analytic property of Ricci-limit spaces. 
As related topics, Gigli (\cite{Gigli}) introduced the spaces of formal differentiable forms on RCD-spaces. 
By the construction, the space of forms become a cochain complex. 
However, its cohomology is not known to be a topological invariant, for Alexandrov spaces. 
It should be remarked that the results obtained in \cite{Honda:ori} are not related to our results yet, because they are stated in terms of Gigli's cohomology.
It should be researched a relationship between them in a near future.

In the present paper, we will consider the topology of Alexandrov spaces, directly.
Compared with general synthetic Riemannian spaces, Alexandrov spaces are very tame topologically and geometrically.
This is a point of the study of Alexandrov spaces. 


\vspace{1em}
\noindent 
{\bf Organization}.
The contents of this paper are the following. 
Sections \ref{sec:prelim}--\ref{sec:top-app} are devoted to study general NB-spaces and their orientability. 
In Section \ref{sec:prelim}, we recall fundamental properties of NB-spaces and prove a few of them. 
We give the precise definitions of orientability for NB-spaces, mentioned in (A) of Theorem \ref{thm:main thm}. 
In Sections \ref{sec:proof of AB} and \ref{sec:proof of C}, we prove Theorem \ref{thm:main thm}. 
In Section \ref{sec:top-app}, we give several topological applications of orientability. 
From Section \ref{sec:Alex}, we begin to consider applications of the orientability to Alexandrov spaces. 
In Section \ref{sec:Alex}, we review the definition of Alexandrov spaces and prove Corollary \ref{cor:bdry}. 
In Section \ref{sec:fundam class cpt}, we discuss about the existence of fundamental classes at each compact set in orientable Alexandrov spaces, using a geometric property of Alexandrov spaces. 
In particular, the duality map from the cohomology with compact support is defined. 
In Section \ref{sec:PD1}, we prove Theorem \ref{thm:PD1}.
In Section \ref{sec:current}, we recall the definition of metric currents introdued by Ambrosio and Kirchheim and give the notion of mass of singluar Lipschitz chains, to prove Theorem \ref{thm:fill rad Alex}.  
Then, we prove Theorem \ref{thm:fill rad Alex}.
Finally, we give three appendices. 
In Appendix \ref{sec:mass}, we prove Theorem \ref{thm:mass} which is used in the proof of Theorem \ref{thm:fill rad Alex}. 
In Appendix \ref{sec:C1-ori}, we give a geometric interpolation of orientability of Alexandrov spaces. 
In Appendix \ref{sec:NB with boundary}, we introduce NB-spaces with boundary and state fundamental properties about their orientability.

\vspace{1em}
\noindent
{\bf Acknowledgement}. 
The author would like to thank Takao Yamaguchi, Shin-ichi Ohta, and Takumi Yokota for discussions and comments. 
He also would like to thank John Harvey for variable comments. 
This work was also supported by JSPS KAKENHI Grant Number 15K17529.



\section{Preliminaries} \label{sec:prelim}
Let us first recall the definiton of arbitrary MCS-spaces. 
\begin{definition}[\cite{Per:Alex2}, \cite{Per:Morse}] \upshape \label{def:MCS}
Let $Y$ be a separable metrizable space. 
We say that $Y$ is a $0$-dimensional MCS-space if it is a discrete space; and say that $Y$ is $n$-dimensional MCS-space when $n \ge 1$ if for any $x \in Y$, there exist an open neighborhood $U$ and a compact $(n-1)$-dimensional MCS-space $\Sigma$ such that $(U,x)$ is homeomorphic to $(c(\Sigma),o)$ as a pointed space.
Such a $U$ is called a conical neighborhood of $x$ and $\Sigma$ is called a generator of the conical neighborhood $U$.
\end{definition}
From the definition, a space is a one-dimensional MCS-space if and only if it is a graph with countable verticies having finite degree at each vertex.
Note that an MCS-space is connected if and only if it is path-connected. 

Let $Y$ be an arbitrary $n$-dimensional MCS-space. 
For every $y \in Y$ and an abelian group $G$, we set 
\begin{align*}
H^i(Y|y;G) &:= H^i(Y, Y \setminus \{y\};G) \text{ and } \\
H_i(Y|y;G) &:= H_i(Y,Y \setminus \{y\};G)
\end{align*}
which are called the $i$-th local cohomology and local homology of $Y$ at $y$, respectively. 
For a subset $A \subset Y$, we use similar teminologies: $H^\ast(Y|A;G)=H^\ast(Y,Y\setminus A;G)$ and $H_\ast(Y|A;G) = H_\ast(Y,Y \setminus A;G)$.
If $\Sigma$ is a generator of a conical neighborhood at $y$, then by the excision, we have 
\[
H^\ast(Y|y) \cong \bar H^{\ast-1}(\Sigma) \text{ and } H_\ast(Y|y) \cong \bar H_{\ast-1}(\Sigma)
\]
for any coefficient group, where the bar symbol means the reduced (co)homology.

Due to Kwun \cite{Kwun}, conical neighborhoods are unique in the sense that if $U$ and $V$ are two conical neighborhoods at the same point $y$, then $(U,y)$ and $(V,y)$ are homeomorphic. 
From this, two generators $\Sigma$ and $\Lambda$ of conical neighborhoods at the same point have the same homotopy type. 
In particular, if $\Sigma$ is connected, so is $\Lambda$.
Remark that, generators are not homeomorphic, in general. 
For instance, let $A$ be the suspension over a Poincar\'e homology $3$-sphere, then due to Cannon-Edwards's double suspension theorem (\cite{C}, \cite{E}) and the result in \cite{Kwun} mentioned just before, $c(A)$ is homeomorphic to $\mathbb R^5 \approx c(S^4)$.
However, $A$ and $S^4$ are not homeomorphic.


From the definition, $Y$ has a canonical stratification into topological manifolds as follows. 
For $0 \le i \le n$, we define $Y[i]$ a subset of $Y$ by 
$Y[i] \ni x$ if and only if there is a conical neighborhood $U$ of $x$ which is homeomorphic to $(\mathbb R^i \times c(\Lambda), o)$, where $\Lambda$ is a compact $(n-i-1)$-dimensional MCS-space. 
Here, $(-1)$-dimensional MCS-space is considered as an empty-set and its cone is a one-point space.
The complement of $Y[i]$ is denoted by $Y_{i-1}$ which gives a filtration 
\[
Y = Y_n \supset Y_{n-1} \supset \dots \supset Y_1 \supset Y_0 \supset Y_{-1} = \emptyset
\]
consisting of closed sets. 
Here, note on how to set the indices. 
From the definition, $Y_i \setminus Y_{i-1}$ is a topological $i$-manifold.
For each $i$, the set $Y^{(i)} := Y_i \setminus Y_{i-1}$ is called the $i$-th {\it stratum} of the filtration $\{Y_j\}$.
The family $\{Y^{(i)}\}$ is called a canonical {\it stratification} of $Y$. 
The $n$-th stratum $Y^{(n)}$ is called the top stratum. 

Let us define $Y_\top \subset Y$ by the condition that $Y_\top \ni x$ if and only if there is an open neighborhood of it which is homeomorphic to $\mathbb R^n$. 
We say that $Y_\top$ is the {\it manifold-part} of $Y$.
From the definition, $Y_\top$ is nothing but $Y^{(n)}$.
Since the following statement is trivial, we omit a proof: 
\begin{lemma} \label{lem:open dense mfd}
$Y_\top$ is dense in $Y$.
\end{lemma}
From this lemma, if $Y_\top$ is connected, then so is $Y$. 
However, the converse is not true in general. 
Indeed, for the cone over a discrete space consisting of three points, its manifold-part is disconnected.

From now on, we are going to study fundamental topological property of NB-spaces. 
Before doing it, we remark about what examples and non-examples of NB-spaces are. 
\begin{example} \label{ex:NB} \upshape
The following are typical examples of NB-spaces.
Let $X$ and $Y$ be compact connected NB-spaces.
\begin{itemize}
\item The cone $c(X)$ is an NB-space with $\dim (c(X)) = \dim X+1$. 
For instance, the cone over a two-torus $T^2$ is an NB-space, but does not admit a metric of Alexandrov space.
\item The join $X \ast Y := X \times Y \times [0,1]/\! \sim$ is a conncted NB-space with $\dim (X \ast Y)=\dim X+\dim Y+1$, where the equivalent relation $\sim$ on $X \times Y \times [0,1]$ is generated by $(x,y,0) \sim (x',y,0)$ and $(x,y,1) \sim (x,y',1)$ for $x,x' \in X$ and $y,y' \in Y$.
\item The suspension $S^0 \ast X$ is a connected NB-space of dimension $\dim X+1$, where $S^0$ denotes a discrete space consisting of only two points.
\item Any non-empty open subset of an NB-space is an NB-space. 
\item The product of two NB-spaces is an NB-space.
\item From the definition, any two-dimensional NB-space is a two-manifold without boundary.
\item Topological manifolds without boundary are NB-spaces. 
\item An Alexandrov $n$-space without boundary is an $n$-dimensional NB-space (Theorem \ref{thm:stab}). 
\item The boundary of an Alexandrov space is an NB-space (Corollary \ref{cor:bdry is NB}). 
\item Regular fibers of semiconcave mapping defined on an open set in an Alexandrov space are NB-spaces (\cite{MY:fiber}).
\end{itemize}
\end{example}

\begin{nonexample} \label{nonex:NB} \upshape
Let us list non-examples of NB-spaces. 
\begin{itemize}
\item Graphs with a vertice of degree at least two are not NB-spaces. For instance, the bouquet of two circles and the cone over it are not NB-spaces. 
\item The cone over the disjoint union of two circles is not an NB-space. 
\item Manifolds with boundary and Alexandrov spaces with boundary are not NB-spaces. 
\item The cone over $S^1 \times [0,1]$ is a non-branching MCS-space with boundary, but is not an NB-space (See also Appendix \ref{sec:NB with boundary}). 
\end{itemize}
\end{nonexample}

From now on, $X$ denotes an $n$-dimensional NB-space.

\begin{lemma} \label{lem:X_top conn}
If $X$ is connected, then $X_\top$ is connected.
\end{lemma}
\begin{proof}
This follows from the induction on the dimension of NB-spaces, a covering argument and the fact that $c(\Sigma)_\top \supset c(\Sigma_\top) \setminus o$ for a compact connected MCS-space $\Sigma$.
\end{proof}

Let us recall that the definition of $H^\ast_\cpt(Z)$ the cohomology of compact support of a space $Z$. 
The set of all compact sets $\{K\}$ in $Z$ turns out a directed set by the inclusion. 
Then, we set $H_\cpt^\ast(Z) := \varinjlim_K H^\ast(Z|K)$.
From the definition, when $Z$ is compact, $H^\ast(Z) = H^\ast_\cpt(Z)$.

\begin{lemma} \label{lem:codim}
The following holds. 
\begin{itemize}
\item[(1)] If $X$ is connected, then $H_\cpt^n(X;\mathbb Z_2) \cong \mathbb Z_2$. 
\item[(2)] Let $\{X_i\}_{i=0}^n$ be a canonical stratification of $X$. 
Then, $X_{n-1}=X_{n-2}=X_{n-3}$. 
In particular, $\dim(S) \le n-3$, where $S = X \setminus X_\top$.
\end{itemize}
\end{lemma}
\begin{proof}
The statements are ture if $\dim X \le 2$, because in such a case, $X$ is a manifold. 
We may assume that $n = \dim X \ge 3$.
First, we prove the statement (2). 
It suffices to prove that $X[n-2] \subset X[n]$. 
Let us take $x \in X[n-2]$. 
Then, there is a $1$-dimensional compact MCS-space $\Gamma$ such that $x$ has a conical neighborhood $U$ homeomorphic to $\mathbb R^{n-2} \times c(\Gamma)$.
We fix a homeomorphism $f : (U,x) \to (\mathbb R^{n-2} \times c(\Gamma), (0,o))$.
Recall that $\Gamma$ is a graph. 
Let us take a vertex $v \in \Gamma$ and set $y = f^{-1}(0,v)$, where $v \in c(\Gamma)$ is regared as a point $(v,1) \in c(\Gamma) = \Gamma \times [0,\infty)/\!\!\sim$.
Then, $y$ has a conical neighborhood $V$ homeomorphic to $\mathbb R^{n-1} \times c(D)$, where $D$ is a discrete space consisting of finitely many points whose cardinarily is the degree of $\Gamma$ at $v$.
So, we have 
\begin{align*}
H^n(X|y) &\cong H^n(V|y) \cong H^n(\mathbb R^{n-1} \times c(D)|(0,o)) \\
&\cong \bar H^0(D) \cong \mathbb Z_2^{\oplus (\sharp D-1)}.
\end{align*}
Here, the cohomology groups are with coefficients in $\mathbb Z_2$.
On the other hands, by the statement (1) of the codimension one case, we have $H^n(X|y) \cong \mathbb Z_2$. 
Hence, the degree at $v$ is two.
This implies that $\Gamma$ is the disjoint union of finitely many circles.
By a similar argment as above, we have 
\[
H^{n}(X|x) \cong H^1(\Gamma).
\] 
Hence, $\Gamma$ is actually a circle, and hence $x \in X[n]$.
The last statement in (2) follows from $S = X_{n-3}$.

A proof of (1) is done along the same line as a proof of a lemma in \cite{GP}. 
By using the fact that the Alexander-Spanier cohomology is isomorphic to the singular cohomology for paracompact Hausdorff spaces (\cite{Sp}), we have an exact sequence 
\[
H^{n-1}(S) \to H_\cpt^n(X_\top) \to H_\cpt^n(X) \to H^n(S).
\]
Here, since $\dim S \le n-2$, we have $H_\cpt^n(X;\mathbb Z_2) \cong H_\cpt^n(X_\top;\mathbb Z_2)$. 
Becasue $X_\top$ is a connected $n$-manifold from Lemma \ref{lem:X_top conn}, we have $H_\cpt^n(X_\top;\mathbb Z_2) \cong \mathbb Z_2$. 
This completes the proof.
\end{proof}

\subsection{Several valid notions of orientability for NB-spaces}

In this subsection, the (co)homologies are with coefficients in $\mathbb Z$, when they are not indicated.
First, we give an alternative proof of Harvey and Searle's result as follows. 
\begin{lemma}[\cite{HS}] \label{lem:HS}
The conditions $(\mathrm{a})$ and $(\mathrm{b})$ in Theorem \ref{thm:main thm} $(\mathrm{A})$ are equivalent for an $n$-dimensional connected NB-space $X$. 
\end{lemma}
\begin{proof}
Since the implication (b) $\Rightarrow$ (a) is trivial, we prove the converse direction. 
Suppose that $X_\top$ is orientable. 
Then, every connected open subset $V$ of $X_\top$ is an orientable manifold. 
In particular, $H_\cpt^n(V) \cong \mathbb Z$ and $H^n(V|y) \to H_\cpt^n(V)$ is isomorphic for every $y \in V$. 
By the excision axiom, the canonical morphism $H_\cpt^n(V) \to H_\cpt^n(X_\top)$ is an isomorphism.
Recall that 
\begin{equation} \label{eq:codim}
H_\cpt^n(X_\top) \to H_\cpt^n(X)
\end{equation}
is an isomorphism, because $\dim S \le n-2$. 
Let $x \in S$ and $U$ a conical neighborhood at $x$. 
Then, $U$ itself is a connected NB-space. 
Therefore, replacing $X$ with $U$ in \eqref{eq:codim}, 
we obtain an isomorphism $H_\cpt^n(U_\top) \to H_\cpt^n(U)$.
Now, we consider the following commutative diagram consisting of canonical morphisms
\[
\xymatrix{
H^n_\cpt(X_\top) \ar[r]^\cong &H_\cpt^n(X) &H^n(X|x) \ar[l] \ar[d]^\cong \\
H_\cpt^n(U_\top) \ar[r]_\cong \ar[u]^\cong &H_\cpt^n(U) \ar[u] &H^n(U|x) \ar[l]^\cong
}
\]
Here, the most right downward arrow is an isomorphism due to the excision axiom. 
Further, the arrow $H_\cpt^n(U) \leftarrow H^n(U|x)$ is isomorphic by the conical structure of $(U,x)$.
Other morphisms indicated as ``$\cong$'' are also isomorphic, by the above observation.
From this diagram, we conclude that $H_\cpt^n(X) \cong \mathbb Z$ and $H_\cpt^n(X) \leftarrow H^n(X|x)$ is isomorphic, for every $x \in S$. 
Hence, $X$ satisifes (b). 
This completes the proof.
\end{proof}

Remark that a torus with a smashed meridian is an MCS-space, but is not an NB-space. 
Such space is orientable in the sense of (a). 
However, it does not satisfies (b), because the generator at the smashed point is the union of disjoint circles.

The following concept is nothing but the condition (c) in Theorem \ref{thm:main thm} (A) which is similar to the condition (b): 
\begin{definition}[cf.\!\!\! \cite{Pet}] \upshape
Let $X$ be an $n$-dimensional NB-space. 
We say that $X$ is {\it locally orientable at} $x \in X$ if $H^n(X|x;\mathbb Z) \cong \mathbb Z$ holds. 
We say that $X$ is {\it cohomologically orientable} if it is locally orientable at every point, and there is a family $\{o^x \in H^n(X|x)\}_{x \in X}$ of generators of local cohomologies such that for any $x \in X$, there exist an open neighborhood $U$ of it and an element $o^U \in H^n(X|U)$ satisfying that $H^n(X|y) \to H^n(X|U)$ is an isomorphism and maps $o^y$ to $o^U$ for every $y \in U$.
Such a family $\{o^x\}_{x \in X}$ is called a {\it cohomological orientation} on $X$.
\end{definition}

\begin{remark} \upshape
Note that, if it is admitted that (e) and (g) are equivalent in the statement (A)$_{n-1}$, then 
$n$-dimensional NB-space $X$ is locally orientable at $x$ if and only if $H_n(X|x;\mathbb Z) \cong \mathbb Z$. 
Due to Theorem \ref{thm:main thm}, it is eventually verified.
\end{remark}

As same as for the case of topological manifolds, we shall employ the following concept for defining orientability of NB-spaces. 
This condition is (d) in Theorem \ref{thm:main thm} (A).
\begin{definition} \upshape
Let $X$ be an $n$-dimensional NB-space. 
We say that $X$ is {\it homologically orientable} if 
$H_n(X|x;\mathbb Z) \cong \mathbb Z$ for any $x \in X$ and 
it admits a family of elements $o_x \in H_n(X|x;\mathbb Z)$ such that $o_x$ is a generator of $H_n(X|x)$ and that for every $x \in X$, there exist an open neighborhood $U$ and an element $o_U \in H_n(X|U)$ such that 
$H_n(X|U) \to H_n(X|y)$ maps $o_U$ to $o_y$ for every $y \in U$.
In this case, the family $\{o_x\}_{x \in X}$ is called a {\it homological orientation} on $X$.
\end{definition}

\begin{remark} \upshape \label{rem:R-ori}
For any commutative unital ring $R$, we can give the notion of $R$-orientability for NB-spaces, corresponding to the conditions (a)--(d) in Theorem \ref{thm:main thm}. 
For instance, if $X$ is orientable, then it is $R$-orientable in any sense, by Theorem \ref{thm:main thm}(B). 
In this terminology, the $\mathbb Z$-orientability is the usual orientability and every NB-space is always $\mathbb Z_2$-orientable.
One can also prove that the statement of Theorem \ref{thm:main thm} replaced with the corresponding statements of $R$-orientability, whenever $R$ is a {\it principal ideal domain}. 
For the reason why $R$ is a PID, see the proof of Theorem \ref{thm:main thm} given in Sections \ref{sec:proof of AB} and \ref{sec:proof of C}.
In particular, if $X$ is $R$-orientable and if $R$ is a principal ideal domain, then one can prove that it is $R/p$-orientable for any prime ideal $p$ of $R$, by using ``$R$-version'' of Theorem \ref{thm:main thm}.
\end{remark}

Finally, let us give a technical remark. 
\begin{remark} \label{rem:a to d is hard} \upshape 
The reason why the proof of (a) $\Rightarrow$ (b) in Lemma \ref{lem:HS} as above goes well is that the condition (b) is written in the terminologies of cohomologies.
That is, we could use the isomorphism \eqref{eq:codim}.
On the other hands, 
the condition (d) is given in the terminologies of homologies.
As after seen, we do not prove the implication (a) $\Rightarrow$ (d) directly.
\end{remark}

\section{Proof of Theorems \ref{thm:main thm} (A) and (B) and \ref{thm:compact}} \label{sec:proof of AB}

This section is devoted to prove (A)$_n$ and (B)$_n$ of Theorem \ref{thm:main thm} and Theorem \ref{thm:compact}. 
Arguments from here are done along the same line as the usual proof of the case of manifolds (See e.g. \cite{May}).

The symbols (a)--(h) denote conditions in Theorem \ref{thm:main thm} (A).
First, we discuss about the conditions (a)--(d).
By the definitions, the implications 
\[
\xymatrix{
(\mathrm{a}) &(\mathrm{b}) \ar@{=>}[d] \ar@{=>}[l] \\
(\mathrm{d}) \ar@{=>}[u] \ar@{=>}[r] &(\mathrm{c}) 
}
\]
are trivial. 
Here, due to Lemma \ref{lem:HS}, we already know that (a) $\Rightarrow$ (b) is true. 
Further, the implication (c) $\Rightarrow$ (b) also holds, by a similar argument to the proof of Lemma \ref{lem:HS}.
(However, we do not need this implication). 
Thus, to prove a half of (A)$_n$, it suffices to show that (c) $\Rightarrow$ (d).
In this section, the (co)homologies are asssumed to be with coefficients in $\mathbb Z$, if they are not indicated. 

\begin{lemma}
Supposing $(\mathrm{B})_{n-1}$, we have the implication $(\mathrm{c}) \Rightarrow (\mathrm{d})$.
\end{lemma}
\begin{proof}
Let $X$ be an $n$-dimensional connected NB-space which is cohomologically orientable. 
Let $\{o^x\}_{x \in X}$ be a cohomological orientation. 
For each $x \in X$, we denote by $\Sigma_x$ a generator at $x$. 
By (B)$_{n-1}$, $\bar H_{n-2}(\Sigma_x)$ has no torsion. 
Hence, the canonical morphism 
\[
H^n(X|x) \to H_n(X|x)'
\]
is an isomorphism. 
Here, $G'$ denotes the group $\mathrm{Hom}_{\mathbb Z}(G,\mathbb Z)$ for an abelian group $G$.
Let us set $(o^x)' \in H_n(X|x)'$ the image of $o^x$ under the map $H^n(X|x) \to H_n(X|x)'$. 
Further, $(o^x)'' \in H_n(X|x)''$ denotes the dual element of $(o^x)' \in H_n(X|x)'$.
By identifying $G$ and $G''$ for $G$ with no torsion, 
we obtain an element $o_x \in H_n(X|x)$ corresponding to $(o^x)''$. 
This gives a homological orientation $\{o_x\}_{x \in X}$ from the construction. 
This completes the proof.
\end{proof}

Let us discuss about the relation between the conditions (a)--(d) and the conditions (e)--(h).
Obviously, $(\mathrm{g}) \Rightarrow (\mathrm{h})$ holds.
By the universal coefficient theorem, $(\mathrm{h}) \Rightarrow (\mathrm{f})$ and $(\mathrm{e}) \Rightarrow (\mathrm{g})$ hold. 
By using the compactness of a space, we have $(\mathrm{b}) \Rightarrow (\mathrm{g})$ for compact connected NB-space. 
As a summary, we have a diagram consisting of implications between the conditions: 
\[
\xymatrix{
(\mathrm{b}) \ar@{=>}[r] &(\mathrm{g}) \ar@{=>}[r] &(\mathrm{h}) \ar@{=>}[d] \\
(\mathrm{d}) \ar@{=>}[u]  &(\mathrm{e}) \ar@{=>}[u] \ar@{-->}[l] &(\mathrm{f}) \ar@{-->}[l]
}
\]
Here, the thick arrows denote implications known already.
The broken arrows denote implications which will be proved from now.

\begin{lemma}
Supposing $(\mathrm{C})_{n}$, we have the implication $(\mathrm{f}) \Rightarrow (\mathrm{e})$.
\end{lemma}
\begin{proof}
Let $X$ be an $n$-dimensional connected NB-space satsfying $H_n(X) \neq 0$. 
For $x \in X_\top$, we have an exact sequence 
\[
H_n(X \setminus \{x\}) \to H_n(X) \to H_n(X|x). 
\]
Since $x \in X_\top$, we have $H_n(X|x) \cong \mathbb Z$. 
We remark that $X \setminus \{x\}$ is a connected non-compact NB-space. 
Due to the vanishing theorem (C)$_n$, we conclude that the map $H_n(X) \to H_n(X|x)$ is injective. 
From the asumption, $H_n(X)$ is regarded as a non-trivial ideal in $\mathbb Z$.
Therefore, we obtain the conclusion: $H_n(X) \cong \mathbb Z$.
\end{proof}


\begin{lemma} \label{lem:fundam}
Supposing $(\mathrm{A})_{n-1}$ and $(\mathrm{C})_n$, 
we obtain $(\mathrm{e}) \Rightarrow (\mathrm{d})$.
Moreover, if $H_n(X;\mathbb Z) \cong \mathbb Z$, then the canonical map $H_n(X;\mathbb Z) \to H_n(X|x;\mathbb Z)$ is an isomorphism for each $x \in X$.
\end{lemma}
\begin{proof}
Let $X$ be an $n$-dimensional connected NB-space satisfying $H_n(X) \cong \mathbb Z$. 
By (C)$_n$, the canonical map 
\begin{equation} \label{eq:inj}
H_n(X) \to H_n(X|x)
\end{equation}
is injective for any $x \in X$. 
Therefore, by (A)$_{n-1}$, we have $H_n(X|x) \cong \mathbb Z$.
Hence, we can choose an element $a \in \mathbb Z \setminus \{0\}$ such that the map \eqref{eq:inj} is identified with $\mathbb Z \ni b \mapsto ab \in \mathbb Z$.
If $a$ is not a unit, then there is a prime divisor $p$ of $a$. 
Let us consider the quotient field $F:= \mathbb Z/(p)$. 
Now, let us consider the following commutative diagram consisting of canonical maps: 
\[
\xymatrix{
H_n(X) \otimes F \ar[r] \ar[d] &H_n(X|x) \otimes F \ar[d] \\
H_n(X;F) \ar[r] &H_n(X|x;F)
}
\]
Because the vanishing theorem (C)$_n$ holds for any coefficients, the bottom rightward arrow is injective. 
The downward arrows are also injective, by a purely algebraic reason. 
However, the top rightward arrow is trivial from the definition of $F$. 
Counting the dimension of the image of $H_n(X) \otimes F$ in $H_n(X|x;F)$, we obtain a contradiction. 
Therefore, the map \eqref{eq:inj} is known to be isomorphic. 
\end{proof}

Summarizing the above lemmas, we have one of implications mentioned in the introduction: 
\begin{conclusion} 
Supposing $(\mathrm{A})_{n-1}$, $(\mathrm{B})_{n-1}$ and $(\mathrm{C})_n$, we obtain $(\mathrm{A})_n$.
\end{conclusion}

Now, we prove another implication mentioned in the introduction: 
\begin{proposition}
Supposing $(\mathrm{A})_n$ and $(\mathrm{C})_n$, we obtain $(\mathrm{B})_n$. 
\end{proposition}
\begin{proof}
Let $X$ be a compact connected $n$-dimensional NB-space which is orientable in the sense that it satisfies one of the conditions of (A). 
By (e) of (A)$_n$, we have $H_n(X;\mathbb Z) \cong \mathbb Z$.
Since $X$ is compact, $H_{n-1}(X;\mathbb Z)$ is a finitely generated abelian group. 
Suppose that there is a non-trivial torsion part of $H_{n-1}(X;\mathbb Z)$. 
Then, we have a direct summand of it, which is isomorphic to $\mathbb Z/(p^e)$, where $p$ is a prime of $\mathbb Z$ and $e \ge 1$.
Let $F=\mathbb Z/(p)$. 
Then, by the universal coeffcient theorem, we have 
\[
H_n(X; F) \cong [H_n(X) \otimes F] \oplus \mathrm{Tor}(H_{n-1}(X), F).
\]
We remark that $\mathrm{Tor}(\mathbb Z/(p^e), F) \cong F$ and recall that $\mathrm{Tor}(A \oplus B, -) \cong \mathrm{Tor}(A,-) \oplus \mathrm{Tor}(B,-)$ holds for any abelian groups $A$ and $B$. 
Therefore, $H_n(X;F)$ has $F$-dimension at least two. 
On the other hands, by (C)$_n$, $H_n(X;F)$ is injectively embedded in $H_n(X|x;F)$. 
However, if $x \in X_\top$, then $H_n(X|x;F) \cong F$. 
This is a contradiction. 
Therefore, we conclude that $H_{n-1}(X;\mathbb Z)$ has no torsion.
This completes the proof.
\end{proof}

Let us give a proof of Theorem \ref{thm:compact} using (B)$_n$, (A)$_n$ and (B)$_{n-1}$.
\begin{proof}[Proof of Theorem \ref{thm:compact}]
Let $X$ be an $n$-dimensional compact connected NB-space which is orientable. 
Let $G$ be an abelian group. 
For each $x \in X$, we consider the following commuttive diagram consisting of canonical maps: 
\[
\xymatrix{
H_n(X;G) \ar[r] \ar[d] &H_n(X) \otimes G \ar[d] \\ 
H_n(X|x;G) \ar[r] &H_n(X|x) \otimes G
}
\]
By (B)$_n$ and (B)$_{n-1}$, the rightward arrows are isomorphic. 
By the last statement of (A)$_n$, the right downward arrow $H_n(X) \otimes G \to H_n(X|x) \otimes G$ is isomorphic. 
Therefore, the rest map $H_n(X;G) \to H_n(X|x;G)$ is isomorphic. 
This completes the proof.
\end{proof}

\begin{remark} \upshape
By a similar argument to the proof of Theorem \ref{thm:compact}, one can obtain the followng: 
if a compact connected NB-space $X$ of dimension $n$ is $R$-orientable for a principal ideal doamin $R$, then $H_n(X;G) \cong G$ and $H_n(X;G) \to H_n(X|x;G)$ is an isomorphism for every $x \in X$, where $G$ is an arbitrary $R$-module.
\end{remark}

\section{Proof of the vanishing theorem} \label{sec:proof of C}
Let us prove the statement (C)$_n$ of Theorem \ref{thm:main thm} 
which is called the vanishing theorem, here. 
We may assume that $n \ge 3$. 
The proof of the vanishing theorem is done by a similar way to the proof of usual vanishing theorem for topological manifolds given in \cite{May}. 
Suppose that (A)$_{n-1}$ and (B)$_{n-1}$ are true.

\begin{lemma} \label{lem:cone01}
Let $\Sigma$ be an $(n-1)$-dimensional compact connected NB-space, which is orientable in the sense of $(\mathrm{A})_{n-1}$.
Then, its cone $C=c(\Sigma)$ is homologically orientable.
Moreover, for any $x \in C$, there is an open set $U$containing $x$ and the apex $o$ such that $H_n(C|U;G) \to H_n(C|y;G)$ is isomorphic for any $y \in U$, where $G$ is an abelian group.
\end{lemma}
\begin{proof}
Let us represent $C$ as $C=[0,\infty) \times \Sigma /\{0\} \times \Sigma$ and an element in $C$ as $t \xi$, for $t \ge 0$ and $\xi \in \Sigma$. 
In this terminology, $0\xi$ denotes the apex $o$ of the cone.
Let us take $t \xi \in C$ with $t > 0$ and $\xi \in \Sigma$ and set $U := \{s \eta \mid \eta \in \Sigma, s < t+1\}$. 
Then, $U$ is open in $C$ and contains $o$ and $t \xi$. 
Let us take $s \eta \in U$. 
If $s > 0$, then we have a commutative diagram 
\[
\xymatrix{
H_n(C|U;G) \ar[r] \ar[d] &H_n(C|s\eta; G) \ar[d]\\
H_{n-1}(\Sigma;G) \ar[r] &H_{n-1}(\Sigma|\eta; G)
}
\]
Here, the downward arrows are taken to be isomorphisms, by using the cone structure. 
By Theorem \ref{thm:compact}, the bottom rightward arrow is an isomorphism. 
Therefore, the top rightward arrow is an isomorphism. 
When $s = 0$, it is trivial that $H_n(C|U;G) \to H_n(C|o;G)$ is an isomorphism, by the cone structure. 
This completes the proof. 
\end{proof}

From now on and in this section, $X$ denotes an $n$-dimensional connected NB-space with $n \ge 3$. 
We denote by $X_\mathrm{ori}$ the set of points at which locally orientable, that is, 
\[
X_\mathrm{ori} := \{x \in X \mid H^n(X|x; \mathbb Z) \cong \mathbb Z \}.
\]
Although we do not use in the proof, we remark that, by the inductive hypothesis, 
\[
X_\mathrm{ori} = \{x \in X \mid H_n(X|x;\mathbb Z) \cong \mathbb Z\} = \{x \in X \mid H_n(X|x;\mathbb Z) \neq 0\}.
\]
\begin{lemma} \label{lem:cone02}
$X_\top \subset X_\mathrm{ori}$ holds. 
In particular, $X_\mathrm{ori}$ has codimension at least three and is open connected.
\end{lemma}
\begin{proof}
The first statement follows from the definition. 
Hence, by Lemma \ref{lem:codim}, $\dim (X \setminus X_{\mathrm{ori}}) \le n-3$.
From Lemma \ref{lem:cone01}, $X_\mathrm{ori}$ is open in $X$. 
Since $X_\top$ is dense and path-connected, $X_\mathrm{ori}$ is path-connected. 
\end{proof}

The vanishing theorem holds for open sets in a cone:
\begin{lemma} \label{lem:cone03}
Let $\Sigma$ be an $(n-1)$-dimensional compact connected NB-space.
Then, for every open set $U$ in the cone $C=c(\Sigma)$, we have $H_n(U;G) = 0$, where $G$ is an abelian group.
\end{lemma}
\begin{proof}
Since $C$ is contractible, we have $\bar H_\ast(C;G) = 0$. 
Hence, $H_{n+1}(C,U;G)$ and $H_n(U;G)$ are isomorphic. 
Here, since $\dim C = n$, we have $H_{n+1}(C, U;G) = 0$. 
This completes the proof.
\end{proof}

From now on, we abbriviate $G$ and write $H_\ast(-;G)=H_\ast(-)$ (in the proofs of statements). 
\begin{lemma} \label{lem:cone04}
Let $U$ and $C$ be as in Lemma \ref{lem:cone03}.
Let $t \in H_n(C,U;G)$. 
If $t$ maps to zero in $H_n(C|x;G)$ for any $x \not\in U$, then $t=0$.
\end{lemma}
\begin{proof}
This follows from the argument of the proof of a lemma at p.158 in \cite{May} and Lemma \ref{lem:cone03}. 
\end{proof}

\begin{lemma} \label{lem:cone05}
Let $X$ be an $n$-dimensional non-compact connected NB-space. 
Then, the canonical map $H_n(X;G) \to H_n(X|x;G)$ is identically zero, for every $x \in X$.
\end{lemma}
\begin{proof}
From the definition, we may assume that $x \in X_\mathrm{ori}$.
Let $D$ be a conical neighborhood of $x$. 
Since $X_\mathrm{top}$ is dense in $X$, there is a point $y \in D \cap X_{\mathrm{top}}$.

Let us take an element $s \in H_n(X)$. 
Then, there exist a compact set $K \subset X$ and an element $s_K \in H_n(K)$ mapping to $s$ under the map $H_n(K) \to H_n(X)$ induced from the inclusion. 
We choose a point $z \in X_\mathrm{top} \setminus K$. 
Since $X_\mathrm{top}$ is connected, there is a simple curve $\gamma$ connecting $y$ and $z$ in $X_\mathrm{top}$. 
Let us take a neighborhood $W$ of $\gamma$ homeomorphic to $\mathbb R^n$. 
Then, in the following diagram 
\[
\begin{CD}
& H_n(W|y) @<<< & H_n(W|\gamma) @>>> & H_n(W|z) \\
& @VVV & @VVV &@VVV \\
& H_n(X|y) @<<< & H_n(X|\gamma) @>>> & H_n(X|z) 
\end{CD}
\]
the two horizontal arrows from $H_n(W|\gamma)$ are isomorphisms. 
The downward three arrows are isomorphic by exicision. 
Therefore, the two horizontal arrows from $H_n(X|\gamma)$ are isomorphic. 
Here, we remenber that $x \in X_{\mathrm{ori}}$. 
By Lemma \ref{lem:cone01}, the canonical maps 
\begin{equation*} 
H_n(X|x) \leftarrow H_n(X|U) \to H_n(X|y)
\end{equation*}
are isomorphic, where $U$ is an open set containing $x$ and $y$ in $D$. 

Since the inclusion $K \to (X,X \setminus z)$ factors $(X \setminus z, X \setminus z)$, the element $s_K$ maps to zero in $H_n(X|z)$. 
Hence, $s$ maps to zero in $H_n(X|z)$. 
Passing the isomorphisms in the above diagrams, 
we know that $s$ maps to zero in $H_n(X|x)$. 
The following diagram might help to understand this argument: 
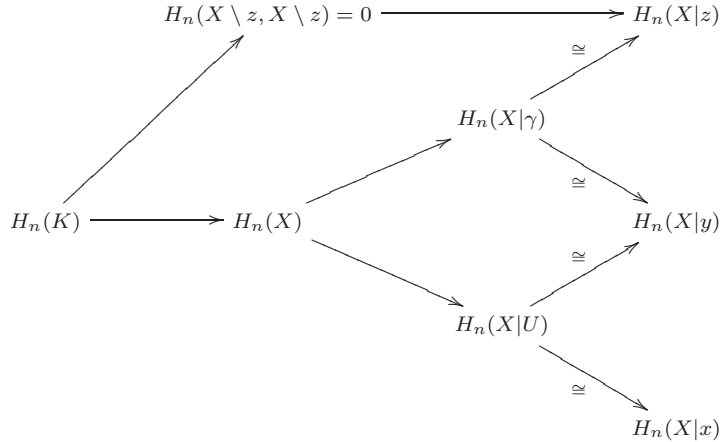
\[
{\tiny 
\xymatrix{
& H_n(X \setminus z, X \setminus z) = 0 \ar[rr] && H_n(X|z) \\
&& H_n(X|\gamma) \ar[ru]^{\cong} \ar[rd]_{\cong} & \\
H_n(K) \ar[ruu] \ar[r] & H_n(X) \ar[ru] \ar[rd] && H_n(X|y) \\
&& H_n(X|U) \ar[ru] ^{\cong} \ar[rd]_{\cong} & \\
&&& H_n(X|x) 
}
}
\]
Therefore, we conclude that the map $H_n(X) \to H_n(X|x)$ is a zero map. 
This completes the proof. 
\end{proof}

\begin{lemma} \label{lem:cone06}
Let $X$ be as in Lemma \ref{lem:cone05} and $U$ and $V$ open subsets in $X$. 
Suppose that $U$ is a conical domain and $V$ satisifies $H_n(V;G)=0$. 
Then, we have $H_n(U \cup V;G)=0$.
\end{lemma}
\begin{proof}
This follows from the same argument as written at p.159--160 in \cite{May} and Lemma \ref{lem:cone05}.
\end{proof}

\begin{proof}[Proof of $(\mathrm{C})_n$ of Theorem \ref{thm:main thm}]
Let $X$ be an $n$-dimensional non-compact connected NB-space. 
Let $s \in H_n(X)$ and $K$ denote a compact subset of $X$ such that $s$ is in the image of $H_n(K)$. 
Let us take a finite family $\{U_1, \dots, U_q\}$ of conical domains such that $\bigcup_{i=1}^q U_i \supset K$. 
Then, by induction and Lemma \ref{lem:cone06}, we have that $H_n(U_1 \cup \dots \cup U_q) = 0$. 
Therefore, we have $s = 0$. 
This completes the proof.
\end{proof}


\section{Topological applications}
\label{sec:top-app}

\subsection{A Poincar\'e-type duality theorem in the highest degree} 

Due to Theorem \ref{thm:main thm} (A), if $X$ is compact connected $n$-dimensional NB-space and is $R$-orientable for a principal ideal domain $R$, then it has a homological $R$-fundamental class $[X]_R \in H_n(X;R)$, that is a generator of $H_n(X;R)$.
The cap product to $[X]_R$ gives a duality map 
\begin{equation} \label{eq:dual}
[X]_R \cap - : H^k(X;G) \to H_{n-k}(X;G)
\end{equation}
for $0 \le k \le n$, where $G$ is an $R$-module. 

We prove Corollary \ref{cor:dual} generalizing to the following statement:
\begin{theorem} \label{thm:dual}
Let $X, R$ and $G$ be as above. 
Then, the duality map \eqref{eq:dual} is an isomorphism for $k=n$.
\end{theorem}
\begin{proof}
This follows from the compatibility of operators: taking the cap product to the fundamental class, the augumentation and the evaluation. 
\end{proof}

\begin{remark} \upshape \label{rem:middle degree}
When $k$ is a middle degree, the map \eqref{eq:dual} is not an isomorphism, in general.
That is, the usual Poincar\'e duality {\it isomorphism} theorem fails for NB-spaces.
Let us consider the following examples of compact connected NB-spaces. 
\begin{itemize}
\item $X = S^0 \ast T^2$; 
\item $Y= S^0 \ast \mathbb RP^3$;
\item $Z = S^0 \ast \mathbb CP^\ell$, where $\ell \ge 2$;
\item $W = S^0 \ast \mathbb RP^2$.
\end{itemize}
Note that $X$ does not admit a metric of an Alexandrov space. 
The last three spaces $Y$, $Z$ and $W$ admit metrics of Alexandrov spaces, and moreover, 
they are obtained as the Gromov-Hausdorff limits of Riemannian manifolds keeping lower curvature bounds (\cite{Yam:CP}). 
Obviously, $Y$ and $W$ regared as a Riemannian orbifold.
The first three spaces $X, Y$ and $Z$ are orientable, and $W$ is $\mathbb Z_2$-orientable.
These spaces do not satisfy the duality isomorphism theorem at middle degrees as follows. 

\begin{itemize}
\item $\dim X= 3$, $H^2(X;\mathbb Q) \cong \mathbb Q^2$ but $H_1(X;\mathbb Q) = 0$; 
\item $\dim Y = 4$, $H^3(Y;\mathbb Z) \cong \mathbb Z_2 \cong H^3(Y;\mathbb Z_2)$ but $H_1(Y;\mathbb Z) = 0 = H_1(Y;\mathbb Z_2)$.
\item $\dim Z = 2 \ell +1$ but $D_Z : H^k(Z;\mathbb Q) \to H_{2\ell+1-k}(Z;\mathbb Q)$ is not isomorphic for $2 \le k \le 2 \ell -1$.
\item $\dim W = 3$, $H^2(W;\mathbb Z_2) \cong \mathbb Z_2$ but $H_1(W;\mathbb Z_2) = 0$.
\end{itemize}

\end{remark}



\subsection{Orientations and geometric constructions} 
Due to Lemma \ref{lem:cone01}, if an NB-space $X$ is compact connected and orientable, then $c(X)$ is orientable. 
Further, we have 
\begin{corollary} \label{cor:ori product}
Let $X$ and $Y$ denote connected NB-spaces. 
Let $R$ be a principal ideal domain. 
Then, the following holds. 
\begin{enumerate}
\item When $X$ is compact, $X$ is $R$-orientable if and only if $c(X)$ is $R$-orientable. 
\item When $X$ is compact, $X$ is $R$-orientable if and only if the suspension $S^0 \ast X$ of $X$ is $R$-orientable. 
\item When both $X$ and $Y$ are compact, both $X$ and $Y$ are $R$-orientable if and only if the join $X \ast Y$ is $R$-orientable. 
\item Both $X$ and $Y$ are $R$-orientable if and only if the product $X \times Y$ is $R$-orientable. 
\end{enumerate}
\end{corollary}
\begin{proof}
The proof of Lemma \ref{lem:cone01} works for proving (1). 
The statements (2), (3) and (4) are proved directly by using Theorem \ref{thm:main thm} (A). 
\end{proof}

Note that $S^k \ast \mathbb RP^2$ is simply-connected, but is not orinetable for $k \ge 0$.
Further, it admits a metric of an Alexandrov space.

\subsection{Another topological interpolation of orientability}
One can consider an ``orientation bundle'' over an NB-space, in a usual way, with a little bit modification.
Let $X$ be a connected NB-space of dimension $n$. 
Let $X_\mathrm{ori}$ denote the subset of $X$ consisting of points at which $X$ is locally orientable. 
Let us set 
\[
\Theta_x := H_n(X|x;\mathbb Z)
\]
for $x \in X$ and consider the disjoint union 
\[
\Theta_X := \bigsqcup_{x \in X} \Theta_x. 
\]
Let us give a topology on $\Theta_X$ as follows. 
Let 
\[
\pi : \Theta_X \to X
\]
be a projection defined by $\pi(\Theta_x) = x$.
We set $\Theta_{X_\ori} := \pi^{-1}(X_\ori)$. 
We equip the topology $\mathcal U_\ori$ on $\Theta_{X_\ori}$ so that $\pi : \Theta_{X_\ori} \to X_\ori$ is a covering map with fiber $\mathbb Z$, using Lemma \ref{lem:cone01}, in a natrual way. 
For $x \in X \setminus X_\ori$, we set 
\[
\mathcal U_x := \{\pi^{-1}(U) \mid U \text{ is a conical neighborhood of }x \}.
\]
Then, the topology on $\Theta_X$ is defined as the topology generated by $\mathcal U_\ori \cup \bigcup_{x \in X \setminus X_\ori} \mathcal U_x$.

From the construction, $\pi : \Theta_{X_\top} \to X_\top$ is the standard orientation bundle of the topological manifold $X_\top$.
Using this generlized orientation bundle $\Theta_X$, we can give another characterization of orientability as follows. 
\begin{corollary}
Let $X$ be a connected NB-space. 
Then, $X$ is orientable if and only if the projection $\pi : \Theta_X \to X$ is a covering map with fiber $\mathbb Z$ and there is a continuous section $s : X \to \Theta_X$ such that $s(x)$ is a generator at each $x \in X$.
Further, if $X$ is orientable, then $\pi$ is a trivial bundle with fiber $\mathbb Z$.
\end{corollary}
In this corollary, the family $\{s(x) \}_{x \in X}$ is nothing but an orientation on $X$.

As a related topic to the concept of our generalized orientation bundles, 
Harvey and Searle (\cite{HS}) gave a concept of ramified orientable double coverings of NB-spaces.
That is, they proved:
\begin{theorem}[{\cite[Theorem 2.4]{HS}}] \label{thm:ram}
For each connected non-orientable NB-space $X$, there is an orientable NB-space $\hat X$ with an involution $\iota$ such that the quotient space $\hat X / \iota$ is homeomorphic to $X$ and that the quotient map $\pi : \hat X \to X$ is a ramified double cover whose ramification locus is $X \setminus X_\mathrm{ori}$.
\end{theorem}
Here, ``$\pi$ is a ramified double cover'' means that each fiber consists of one or two elements, and 
the rafimication locus of $\pi$ is the set of $x \in X$ such that $\pi^{-1}(x)$ is a set of a single point. 
Let us recall the construction of $\hat X$ in our terminology. 
Let $\pi : \Theta_X \to X$ be a generalized orientation bundle obtained in the previous subsection. 
Let us set 
\[
\Theta_x^0 := \{u \in H_n(X|x;\mathbb Z) \mid u \text{ is a generator} \}
\]
for every $x \in X_\mathrm{ori}$ and set $\Theta_x^0 := \Theta_x$ otherwise. 
Then, the restriction of $\pi$ to $\hat X := \bigcup_{x \in X} \Theta_x^0$ satisfies the desired condition.
Here, $\iota$ is defined in a natural way: $\iota(u) = -u$.

The spherical join $\mathbb S^k \ast \mathbb RP^2$ of the unit sphere $\mathbb S^k$ and the real projective plane $\mathbb RP^2$ is a typical non-orientable closed Alexandrov space, and hence, it is an NB-space. 
The ramified orientable double cover of it is a $(k+3)$-sphere.

\subsection{Further topological applications}

\begin{corollary}
If $X$ is a non-compact connected NB-space of dimension $n$, then 
\begin{enumerate}
\item the torsion-part of $H_{n-1}(X;\mathbb Z)$ is trivial; and 
\item $H^n(X;G)=0$ for any abelian group $G$.
\end{enumerate}
\end{corollary}
\begin{proof}
There is an exact sequence as the following form (for instance, see \cite{Bre}): 
\[
0 \to H_n(X;\mathbb Z) \otimes \mathbb Q/ \mathbb Z \to H_n(X; \mathbb Q / \mathbb Z) \to TH_{n-1}(X;\mathbb Z) \to 0
\]
where $TH_{n-1}(X;\mathbb Z)$ 
denotes the torsion-part of $H_{n-1}(X;\mathbb Z)$. 
By the vanishing theorem (Theorem \ref{thm:main thm} (C)), we obtain the first conclusion. 

Since $H_n(X;\mathbb Z)=0$, we have 
\[
H^n(X;G) \cong 
\mathrm{Ext}_{\mathbb Z}(H_{n-1}(X; \mathbb Z), G).
\]
By (1), we obtain the second conclusion.
\end{proof}

\begin{corollary} \label{cor:non-ori}
Let $X$ be an $n$-dimensional connected compact NB-space. 
If it is non-orientable, then the following holds. 
\begin{enumerate}
\item $H_n(X;G) \cong 
\{g \in G \mid 2 g = 0\}$; 
\item The torsion subgroup of $H_{n-1}(X;\mathbb Z)$ is isomorphic to $\mathbb Z_2$; and 
\item $H^n(X;G) \cong G / 2 G$.
\end{enumerate}
\end{corollary}
\begin{proof}
We prove (1). 
Since $X$ is not orientable, due to Lemma \ref{lem:HS}, its manifold-part $X_\top$ is not orientable. 
Therefore, there is a loop $\gamma$ in $X_\top$ which reverses orientations. 
That is, the loop induces a map 
\[
\gamma_\ast : H_n(X|x;\mathbb Z) \to H_n(X|x;\mathbb Z)
\]
so that 
\[
\gamma_\ast = -1
\]
where $x$ is a point in $\gamma$.
Hence, we obtain the conclusion along the same line as the proof of the case of manifolds, 
by using Theorem \ref{thm:main thm} (C).

Due to (1), we have 
\[
H_n(X;\mathbb Z_a) \cong 
\left\{ 
\begin{aligned}
&\mathbb Z_2 &&\text{if } a \text{ is even}\\
&0 &&\text{if } a \text{ is odd}
\end{aligned}
\right.
\]
By the universal coefficient thereom, we obtain the conclusion of (2).

The claim (3) follows from (2).
This completes the proof.
\end{proof}



\begin{corollary} \label{cor:Z=R}
The $\mathbb Z$-orientability and the $\mathbb K$-orientability are equivalent for NB-spaces, where $\mathbb K$ is a field of characteristic $p$ of $p \ge 3$ or $p=0$.
\end{corollary}
\begin{proof}
We may assume that spaces are connetced. 
Let $X$ be a connected NB-space. 
If $X$ is ($\mathbb Z$-)orientable, then it is $\mathbb K$-orientable as mentioned in Remark \ref{rem:R-ori}.
Suppose that $X$ is $\mathbb K$-orientable. 
When $X$ is compact, we have $H_n(X;\mathbb K) \cong H_n(X;\mathbb Z) \otimes \mathbb K$, by Corollary \ref{cor:non-ori}. 
Therefore, by Theorem \ref{thm:main thm}, we obtain the conclusion. 
For the case that $X$ is non-compact, applying Theorem \ref{thm:main thm} to the local homology at each point, we also obtain the conclusion.
This completes the proof.
\end{proof}

\section{Reviewing Alexandrov spaces and simple applications} \label{sec:Alex}

\subsection{Alexandrov spaces} \label{subsec:Alex}
We brefly recall the definition of Alexandrov spaces and its fundamental properties. 
An isometric embedding of an interval to a metric space is called a geoedesic in the metric space. 
A metric space is said to be {\it geodesic} if for any two points in the space, there exists a geodesic connecting the two points. 

To define Alexandrov spaces, we recall the notion of comparison angle. 
For $\kappa \in \mathbb R$, let $\mathrm{sn}_\kappa$ be a function of one-parameter defined as the unique solution to the ODE
\[
\mathrm{sn}_\kappa'' + \kappa\, \mathrm{sn}_\kappa=0, \hspace{1em} \mathrm{sn}_\kappa(0)=0, \hspace{1em} \mathrm{sn}'_\kappa(0)=1.
\]
Further, we set 
\[
\mathrm{cs}_\kappa(u) := \mathrm{sn}_\kappa'(u). 
\]
For three points $a,b,c$ in a metric space with $0 < \min\{d(a,b), d(a,c)\}$ (and $d(a,b)+d(b,c)+d(c,a) < 2 \pi / \sqrt{\kappa}$ when $\kappa > 0$), the $\kappa$-comparison angle $\tilde \angle_\kappa bac$ at $a$ is defined by 
\[
\tilde \angle_\kappa bac := \arccos \frac{\mathrm{cs}_\kappa (d(b,c))-\mathrm{cs}_\kappa (d(a,b))\,\mathrm{cs}_\kappa (d(a,c))}{\kappa\,\mathrm{sn}_\kappa (d(a,b))\,\mathrm{sn}_\kappa (d(a,c))}
\]
when $\kappa \neq 0$, and 
\[
\tilde \angle_0 bac := \lim_{\kappa \to 0} \tilde \angle_\kappa bac.
\]
\begin{definition}[\cite{BGP} cf. \cite{BBI}] \label{def:Alex} \upshape
A complete metric space $M$ is called an {\it Alexandrov space} if it is a geodesic space and that for each $x \in M$ there exist $\kappa \in \mathbb R$ and an open neighborhood $U$ of $x$ such that for any $a \in U$ and $b,c,d \in U \setminus \{a\}$, we have 
\[
\tilde \angle_\kappa bac + \tilde \angle_\kappa cad + \tilde \angle_\kappa dab \le 2 \pi. 
\]
\end{definition}
For other equivalent definitions of Alexandrov spaces, we refer \cite{BBI} and \cite{BGP}.
Another reformulation of the defintion is discussed in \cite{Pet:qg}.

From the definition, any Alexandrov space $M$ is connected. 
It is known that the Hausdorff dimension of $M$ equals to the Lebesgure covering dimension of it (\cite{BGP}, \cite{Pla}), which is called the {\it dimension} of $M$.
We only deal with finite-dimensional Alexandrov spaces, in the present paper.

We say that an Alexandrov space $M$ is said to be {\it of curvature $\ge \kappa$} for a real number $\kappa$ if $\dim M \ge 2$ and $M$ has curvatue $\ge \kappa$ at any point. 
As a consequnce of globalization theorem (\cite{BGP}), we know:
\begin{theorem}[Bonnet-type theorem \cite{BGP}] \label{thm:Bonnet}
If $M$ is an Alexandrov space of curvature $\ge \kappa$ for $\kappa > 0$ and of dimension not less than two, then its diameter is not greater than $\pi / \sqrt{ \kappa}$. 
\end{theorem}
Taking into account this theorem (to obtain a compatibility to Theorem \ref{thm:space of directions}), we use the convention that a one-dimensional Alexandrov space is said to be of curvature $\ge \kappa$ for $\kappa > 0$ if and only if its diameter is less than or equal to $\pi / \sqrt{\kappa}$.

For two non-trivial geodesics $c_1, c_2$ emenating the same point $x$ in an Alexandrov space $M$, the notion of the angle $\angle(c_1, c_2)$ between them is well-defined (see \cite{BBI}, \cite{BGP}). 
The angle $\angle$ becomes a psuedo-distance function on the set of all non-trivial geodesics starting at $x$. 
By a standard way, we obtain a metric space as the quotient of this psuedo-distance. 
Let us denote by $\Sigma_x$ its completion with respect to the metric $\angle$, which is called the space of directions at $x$. 

\begin{theorem}[\cite{BGP}] \label{thm:space of directions}
If $M$ is an Alexandrov space of dimension $n$ for $n < \infty$ and $x \in M$, then $\Sigma_x$ is an Alexandrov space of curvature $\ge 1$ when $n \ge 2$; is the metric space consisting one point or two points with distance $\pi$ when $n = 1$; or is the empty-set when $n=0$. 
\end{theorem}

Now we recall one of important results in the theory of Alexandrov spaces, obtained by Perelman. 
\begin{theorem}[Perelman's stability theorem \cite{Per:Alex2} cf. \cite{Kap}] \label{thm:stab}
For any $n$-dimensional Alexandrov space $X$ and a point $x \in X$, there exists $r > 0$ 
such that 
$(B_r(x),x)$ is homeomorphic to $(c(\Sigma_x), o)$.
Here, $B_s(y)$ denotes the open metric ball of radius $s$ centered at $y$ in a metric space.
\end{theorem}

A one-dimensional Alexandrov space is nothing but a complete Riemannian one-manifold possibly with boundary. 
The notion of boundary of an arbitrary Alexandrov space 
is inductively defined as follows. 
The {\it boundary} of an $n$-dimensional Alexandrov space $X$ with $n \ge 2$ is defined as the subset $\partial X$ of $X$ consisting of all $x \in X$ such that $\Sigma_x$ has non-empty boundary. 
For an open subset $M$ of an Alexandrov space $X$, we set 
\[
\partial M:= M \cap \partial X.
\]
From Theoerm \ref{thm:stab}, if such an $M$ does not meet the boundary, that is, $\partial M = \emptyset$, then
it is an NB-space.
Note that, in our notation, the symbol $\partial M$ does not indicate the topological boundary of $M$ as a subset of $X$.

Further, Perelman gave a result saying about a topological structure of positvely curved Alexandrov spaces with boundary: 
\begin{theorem}[\cite{Per:Alex2}] \label{thm:positive boundary}
Let $\Sigma$ be an $n$-dimensional Alexandrov space of positive curvature with $n \ge 1$. 
Suppose that $\Sigma$ has non-empty boundary.
Then, $(\Sigma, \partial \Sigma)$ is homeomorphic to $(\bar c(\partial \Sigma), \partial \Sigma \times \{1\})$. 
\end{theorem}
Here, $\bar c(A)$ denotes a closed cone over $A$ which is defined as $A \times [0,1] / A \times \{0\}$ and $A \times \{t\}$ is regarded as a subset of $\bar c(A)$ in a canonical way.
As a direct corollary to this theorem, we have a topological chracterization of the boundary as follows. 
For an Alexandrov $n$-space $M$, 
\[
\partial M = \{x \in M \mid H_n(M|x;\mathbb Z_2) = 0\}.
\]

Further, from Theorem \ref{thm:positive boundary}, we obtain the following:
\begin{corollary}[\cite{Yam:4-dim}] \label{cor:collar}
The boundary of any Alexandrov space has a collar neighborhood.
\end{corollary}
We will prove Corollary \ref{cor:collar} by generalizing to Proposition \ref{prop:collar} later.

Combinig two Theorems \ref{thm:stab} and \ref{thm:positive boundary}, we obtain: 
\begin{corollary} \label{cor:bdry is NB}
The boundary of any Alexandrov space of dimension $n$, with $n \ge 2$, is an $(n-1)$-dimensional NB-space. 
\end{corollary}
We remark that the boundary of an Alexandrov space with the intrinsic metric is not known to be an Alexandrov space.

Fianlly, we fix notations denoting (almost) regular sets of Alexandrov spaces. 
Let $M$ be an open set of an $n$-dimensional Alexandrov space and $\delta > 0$. 
We set 
\[
M_\delta :=
\{ x \in M \mid \Hau^{n-1}(\Sigma_x) / \Hau^{n-1}(\mathbb S^{n-1}) > 1 - \delta \}
\]
which is called the set of $\delta$-regular points. 
Here, $\mathbb S^k$ denotes the unit $k$-sphere of constant curvature one.
Further, we set 
\[
M_0 = \bigcap_{\delta > 0} M_\delta 
\]
which is called the set of regular points. 
Recall that the Bishop-type theorem for positively curved spaces (\cite{BGP}) state the following absolute bound hold: 
\[
\Hau^{n-1}(\Sigma_x) / \Hau^{n-1}(\mathbb S^{n-1}) \le 1.
\]
The notation $M_\delta$ will be used in Appendices \ref{sec:mass} and \ref{sec:C1-ori}.

\subsection{Group actions and orientations}

\begin{theorem} \label{thm:G-action}
Let $M$ be an orientable closed Alexandrov space.
Let $G$ be a finite group of isometries on $M$. 
Suppose that $g : M \to M$ is orientation-preserving for every $g \in G$.
Then, the quotient space $M / G$ is an orientable closed Alexandrov space.
\end{theorem}
\begin{proof}
When $\dim M = 1$, $M / G$ must be a circle. 
So, the statement is trivial in this case. 
We assume that $n= \dim M \ge 2$.
From the definition, the isomorphism on $H^n(M;\mathbb Z)$ induced by every $g \in G$ is the identity map.
Let $F$ denote the set of fixed points in $M$ with respect to the $G$-action.
Then, we obtain 
\begin{align*}
H^n(M/G;\mathbb Q) &\cong H_\cpt^n((M \setminus F)/G;\mathbb Q) \cong H_\cpt^n(M \setminus F;\mathbb Q)^G \\
&\cong H^n(M;\mathbb Q)^G = H^n(M;\mathbb Q) \cong \mathbb Q.
\end{align*}
Here, $A^G$ denotes the $G$-invariant part of an abelian group $A$ with a linear $G$-action.
By Corollary \ref{cor:Z=R} and Theorem \ref{thm:main thm}, we have the conclusion. 
\end{proof}

The isometry group of an Alexandrov space is known to be a Lie group (\cite{FY}).
The orientability gives an obstruction to the dimensions of Lie groups acting on Alexandrov spaces.
\begin{corollary}
Let $M$ be an orientable closed Alexandrov space and $G$ a Lie group acting on $M$ by isometries preserving orientations.
If $M/G$ has the boundary, then $G$ has dimension at least one.
\end{corollary}

\section{Fundamental classes at compact subsets} \label{sec:fundam class cpt}
In this section, we prove the following: 
\begin{theorem} \label{thm:fundam class at K}
Let $M$ be an open subset of an Alexandrov space of dimension $n$ which does not meet the boundary and $R$ a principal ideal domain.
If $M$ is $R$-orientable, then for any compact set $K \subset M$, there is an $R$-fundamental class $z_K \in H_n(M|K;R)$ of $M$ at $K$, 
that is, for each $x \in K$, the canonical morphism $H_n(M|K;R) \to H_n(M|x;R)$ maps $z_K$ to the given local orientaion at $x$.
\end{theorem}

\begin{remark} \upshape 
Since the proof of Theorem \ref{thm:fundam class at K} is based on the following Theorem \ref{thm:good} that was proved by using geometry of Alexandrov spaces, 
the proof does not work for general NB-spaces.
Moreover, it seems that any proof of 
Theorem \ref{thm:fundam class at K}
for the classical case of manifolds rely on geometric properties ``the intersection of two convex subsets in $\mathbb R^n$ is convex'' and ``convex subsets in $\mathbb R^n$ are conical''.
The following Theorem \ref{thm:good} roles a counter-part of such a geometric property in the geometry of Alexandrov spaces.
The author does not care whether our statement in Theorem \ref{thm:fundam class at K} holds or not for general NB-spaces.
\end{remark}
\begin{theorem}[\cite{MY:good}] \label{thm:good}
Let $M$ be an open subset of an Alexandrov space. 
For any open covering $\mathcal V$ of $M$, there is an open covering $\mathcal U$ of $M$ such that $\mathcal U$ is a locally finite refinement of $\mathcal V$ and any finite non-empty intersection of members in $\mathcal U$ is conical.
\end{theorem}
A covering $\mathcal U$ obtained in the theorem is said to be {\it good}, in this paper. 
Actually, in \cite{MY:good}, a geometric property stronger than the statement of Theorem \ref{thm:good} was proved.
However, we need only the topological property stated in Theorem \ref{thm:good} to obtain results in this paper. 

\begin{lemma} \label{lem:ho}
Let $M$ be as in Theorem \ref{thm:fundam class at K}.
Then, the following conditions are equivalent: 
\begin{itemize}
\item[(1)] $M$ is (homologically) $R$-orientable, 
\item[(2)] There exist an open covering $\{U_\alpha\}_\alpha$ of $M$ and a family of local fundamental classes $\{z_\alpha \in H_n(M|U_\alpha)\}_\alpha$ such that $z_\alpha$ and $z_\beta$ map to the same element in $H_n(M|U_\alpha \cap U_\beta)$ for every $\alpha$ and $\beta$. 
\end{itemize}
\end{lemma}
\begin{proof}
Since the implication (2) $\Rightarrow$ (1) is trivial, 
we prove (1) $\Rightarrow$ (2).
We assume that $M$ is homologically orientable and take a homological orientation $\{o_x \in H_n(M|x)\}_{x \in M}$.
By the definition, there exist an open covering $\{U_i\}$ of $M$ and a family of elements $\{o_i \in H_n(M|U_i)\}$ such that 
for any $x \in M$, there is $i$ such that $x \in U_i$ and the canonical morphism $H_n(M|U_i) \to H_n(M|y)$ maps $o_i$ to $o_y$ for any $y \in U_i$. 
By Theorem \ref{thm:good}, there is a good covering $\{V_j\}$ which is a refinement of $\{U_i\}$. 
By retaking a set of indicies of a subfamily of the covering $\{U_i\}$, we may assume that $V_j \subset U_j$ holds for any $j$.
Let $\bar o_j \in H_n(M|V_j)$ denote the image of $o_j$ under the map $H_n(M|U_j) \to H_n(M|V_j)$. 
Let us take indicies $j, k$ with $V_j \cap V_k \neq \emptyset$. 
Since $V_j \cap V_k$ is conical, there is a point $z \in V_j \cap V_k$ corresponding to the vertex of a cone. 
Then, the map $H_n(M|V_j \cap V_k) \to H_n(M|z)$ is an isomorphism.
The inverse image of $o_z$ via the above isomorphism is denoted by $\bar o_{jk} \in H_n(M|V_j \cap V_k)$.
By the compatibility of the family $\{o_x\}_{x \in M}$, the morphisms $H_n(M|V_j) \to H_n(M|V_j \cap V_k) \leftarrow H_n(M|V_k)$ map the elements $\bar o_j$ and $\bar o_k$ to the same element $\bar o_{jk}$.
This completes the proof. 
\end{proof}

\begin{proof}[Proof of Theorem \ref{thm:fundam class at K}]
Since the condition (2) in Lemma \ref{lem:ho} is the same condition for defining orientations for manifolds, written in the book \cite[p.156]{May}, a smilar argument of the proof of the second theorem in the same page of \cite{May} works for proving Theorem \ref{thm:fundam class at K} (noticing that our spaces admitting conical neighborhoods which are not homeomorphic to Euclidean spaces). 
This completes the proof. 
\end{proof}

\begin{remark} \upshape
Due to \cite{MY:good}, 
the boundary of an Alexandrov space admits a good covering in the sense of this paper. 
Therefore, Theorem \ref{thm:fundam class at K} holds for the boundary.
\end{remark}

For a space $M$ as in Theorem \ref{thm:fundam class at K}, 
taking the cap product to the fundamental class at each compact set, and taking the direct limit, we obtain a natural duality map 
\begin{equation} \label{eq:dual0}
D_M : H_\cpt^k(M;G) \to H_{n-k}(M;G)
\end{equation}
for any $R$-module $G$ and $k$ with $0 \le k \le n$. 

The cone $M=c(\mathbb RP^2)$ over the real projective plane is an Alexandrov $3$-space, that is $\mathbb Z_2$-orientable from the definition. 
Then, we have $H_\cpt^2(M;\mathbb Z_2) \cong \mathbb Z_2$ and $H_1(M;\mathbb Z_2) = 0$. 
Hence, the duality map \eqref{eq:dual0} is not isomorphic in general.
Compare this fact with Theorem \ref{thm:PD1} (or Theorem \ref{thm:dual1} (D)).

\begin{theorem} \label{thm:dual0}
Let $M$ be as in Theorem \ref{thm:fundam class at K}.
Then, the duality map 
\[
D_M : H_\cpt^n(M;G) \to H_0(M;G)
\]
is an isomorphism, for any $R$-module $G$.
\end{theorem}
\begin{proof}
When $M$ is compact, the statement is already proved as Theorem \ref{thm:dual}. 
For the case that $M$ is non-compact, a similar argument to the proof of Theorem \ref{thm:dual1} (D) works for proving the case.
\end{proof}

\section{Proof of Theorem \ref{thm:PD1} and Corollary \ref{cor:positive}} \label{sec:PD1}

To prove Theorem \ref{thm:PD1}, it suffices to prove the following: 

\begin{theorem} \label{thm:dual1}
Let $M$ be an open subset of an orientable Alexandrov $n$-space which does not meet the boundary, where $n \ge 2$. 
Then the following holds: 
\begin{itemize}
\item[(D)$_n$] The canonical duality map 
\[
D_M : H^{n-1}_\cpt(M;\mathbb Q) \to H_1(M;\mathbb Q)
\] 
is isomorphic.
\item[(E)$_n$] If $M$ is an Alexandrov space of curvature $\ge \kappa$ for $\kappa > 0$, then $H^{n-1}(M;\mathbb Q) = 0$.
\end{itemize}
\end{theorem}
Note that the statement (E) is similar to 
Corollary \ref{cor:positive}. 
\begin{proof}
The properties (D)$_2$ and (E)$_2$ are known.
Let us assume that $n \ge 3$. 
The proof consists of several steps (Claims \ref{claim:D0}--\ref{claim:D3}). 
\begin{claim} \label{claim:D0}
Supposing $\mathrm{(D)}_n$, we have $\mathrm{(E)}_n$.
\end{claim}
\begin{proof}
Let $\Sigma$ be a positively curved orientable closed Alexandrov $n$-space. 
Due to Bonnet-type thereom (Theorem \ref{thm:Bonnet}), $H_1(\Sigma;\mathbb Z)$ is a finite abelian group. 
Hence, $H_1(\Sigma;\mathbb Q) = 0$. 
Using (D)$_n$, we obtain the conclusion: $H^{n-1}(\Sigma;\mathbb Q)=0$.
\end{proof}

Now, we suppose (E)$_{n-1}$ and prove (D)$_n$. 
Further, $H_\ast(-;\mathbb Q)$ is always denoted by $H_\ast(-)$ until the end of the proof.

\begin{claim} \label{claim:D1}
If $M$ is conical, $M$ satisfies $\mathrm{(D)}_n$.
\end{claim}
\begin{proof}
Since $M$ is conical, there is a point $x \in M$ such that $M$ is homeomorphic to $c(\Sigma_x)$. 
From the assumption, $\Sigma_x$ is orientable. 
By (E)$_{n-1}$, we have $H^{n-1}_\cpt(M) \cong H^{n-2}(\Sigma_x) = 0$.
On the other hands, $H_1(M) = 0$ because $M$ is contractible. 
\end{proof}

\begin{claim} \label{claim:D2}
Let $U$ and $V$ be open subsets of $M$. 
Suppose that $U$, $V$ and $U \cap V$ satisfy $\mathrm{(D)}_n$. 
Then, $U \cup V$ satisfies $\mathrm{(D)}_n$. 
\end{claim}
\begin{proof}
Let us see the following diagram.
\[
\hspace{-6em}
{\tiny
\xymatrix{
&H_\cpt^{n-1}(U \cap V) \ar[r] \ar[d]^{D_{U \cap V}} &H_\cpt^{n-1}(U) \oplus H_\cpt^{n-1}(V) \ar[r] \ar[d]^{D_U \oplus D_V} &H_\cpt^{n-1}(U \cup V) \ar[r] \ar[d]^{D_{U \cup V}} &H_\cpt^{n}(U \cap V) \ar[r] \ar[d]^{D_{U \cap V}} &H_\cpt^{n}(U) \oplus H_\cpt^{n}(V) \ar[d]^{D_U \oplus D_V} \\
&H_1(U \cap V) \ar[r] &H_1(U) \oplus H_1(V) \ar[r] &H_1(U \cup V) \ar[r] &H_0(U \cap V) \ar[r] &H_0(U) \oplus H_0(V) 
}
}
\]
This diagram commutes, because of the naturality of the duality maps.
The left two downward arrows are isomorphisms from the assumption of the sublemma. 
The right two downward arrows are isomorphisms, by Theorem \ref{thm:dual0}.
Since horizontal sequences are exact, by the five lemma, $D_{U \cup V} : H^{n-1}_\cpt(U \cup V) \to H_1(U \cup V)$ is an isomorphism.
\end{proof}

By the definition of the singular homology and an algebraic argument, we have:
\begin{claim} \label{claim:D3} 
If $\{U_i\}_i$ is a totally ordered set of open subsets of $M$ such that $(\mathrm{D})_n$ holds for each $U_i$, then $\bigcup_i U_i$ satisfies $(\mathrm{D})_n$.
\end{claim}

By Theorem \ref{thm:good}, $M$ admits a good covering $\{V_i\}_{i \in \mathbb N}$. 
Setting $U_i = \bigcup_{j \le i} V_j$, we obtain a totally ordered family $\{U_i\}_i$.
Since $U_1$ is conical, by Claim \ref{claim:D1}, $U_1$ satisfies (D)$_n$.
By the definition of good coverings and by Claim \ref{claim:D2}, every $U_i$ satisifes (D)$_n$.
Therefore, by Claim \ref{claim:D3}, $M$ itself satisfies (D)$_n$. 
This completes the proof of Theorem \ref{thm:dual1}.
\end{proof}

Let us give a proof of Corollary \ref{cor:positive}. 
\begin{proof}[Proof of Corollary \ref{cor:positive}]
Let $\Sigma$ be a closed Alexandrov space of positve curvature and of dimension $n \ge 2$. 
First, we assume that $\Sigma$ is orientable. 
By Theorem \ref{thm:dual1} (E), we obtain $H_{n-1}(\Sigma;\mathbb Q) = 0$. 
On the other hands, we know that $H_{n-1}(\Sigma; \mathbb Z)$ has no torsion due to Theorem \ref{thm:main thm} (B). 
Hence, we conclude that $H_{n-1}(\Sigma;\mathbb Z) = 0$. 

Next, we suppose that $\Sigma$ is non-orientable. 
By Theorem \ref{thm:ram}, we have a ramified orientable doouble cover $\pi : \hat \Sigma \to \Sigma$ of $\Sigma$. 
Then, the restriction to $\pi^{-1}(\Sigma_\top)$ of $\pi$ is a usual double covering of $\Sigma_\top$. 
Further, $\pi^{-1}(\Sigma_\top)$ is a connected orientable $n$-manifold. 
Therefore, we have 
\begin{align*}
H^{n-1}(\Sigma;\mathbb Q) &\cong H^{n-1}_\cpt(\Sigma_\top;\mathbb Q) \\
&\cong 
H^{n-1}_\cpt(\pi^{-1}(\Sigma_\top);\mathbb Q)^{\mathbb Z_2} \cong H^{n-1}(\hat \Sigma;\mathbb Q)^{\mathbb Z_2}. 
\end{align*}
Here, $A^G$ denotes the $G$-invariant part of an abelian group $A$ with a group $G$ acting $A$, and the involutive action on $\tilde X$ induces an involutive action on its cohomology. 
Here, we remark that, due to \cite{HS}, $\hat \Sigma$ is known to have a lifted metric which is a metric of an Alexandrov space of the same lower curvature bound as $\Sigma$. 
Hence, from the case that the space is orientable, we obtain the conclusion in the case that the space is non-orientable. 
This completes the proof.
\end{proof}

By the proofs of Corollary \ref{cor:positive} and Claim \ref{claim:D0} in the case that spaces are orientable, we immediately obtain: 
\begin{corollary} \label{cor:finite 1st homology}
Let $n \ge 2$ and $X$ be a closed orientable Alexandrov $n$-space such that $\pi_1(X)$ is finite. 
Then, we have $H_{n-1}(X;\mathbb Z)=0$. 
\end{corollary}



\subsection{Applications to low-dimensional Alexandrov spaces}

\begin{corollary} \label{cor:4-dim PD}
Any positively curved orientable closed Alexandrov $3$-space $\Sigma$ is a ratinal homology sphere. 
Any $4$-dimensional orientable Alexandrov domain $X$ which does not have boundary points is a rational homology manifold. 
In particular, the duality map of $X$ with coefficinets in rational numbers is isomorphic at every degree.
\end{corollary}
\begin{proof}
Let $\Sigma$ be a closed Alexandrov 3-space of positive curvature. 
By Theorem \ref{thm:PD1}, we have 
\[
H_\ast(\Sigma;\mathbb Q) \cong H_\ast(S^3;\mathbb Q).
\]
This is the conclusion of the first statement. 
Hence, if $U$ is the cone over $\Sigma$, we obtain 
\[
H_\cpt^k(U) = H_{4-k}(U) = 0
\]
for $k=1,2$.
So, the Poincar\'e duality isomorphism theorem holds for the case that spaces are cones.
Therefore, Corollary \ref{cor:4-dim PD} follows from an argumet similar to the proof of Theorem \ref{thm:PD1}. 
\end{proof}

\begin{remark} \label{rem:4-dim PD} \upshape 
Corollary \ref{cor:4-dim PD} is already known. 
Indeed, Galaz-Garcia and Guijarro (\cite{GG}) classified the topologies of positively curved Alexandrov 3-spaces.
Due to their work and Theorem \ref{thm:stab}, the space as in the corollary is known to be topologically a 4-dimensional orientable orbifold without boundary. 
It is well-known that orientable orbifolds satisfy the usual Poincar\'e duality isomorphism theorem (\cite{Sa}).
\end{remark}


\section{Metric currents and applications} \label{sec:current}

\subsection{Metric currents} 
Let us recall the notion of ``currents in metric spaces'' introduced by Ambrosio and Kirchheim (\cite{AK}), brefly. 
We also refer to \cite{Lan}.
For metric spaces $X$ and $Y$, we denote by $\Lip(X,Y)$ the set of all Lipschitz maps from $X$ to $Y$. 
Further, we set $\Lip(X) := \Lip(X,\mathbb R)$ and denote by $\Lip_\mathrm{b}(X)$ the subspace of $\Lip(X)$ consisting of all bounded functions.
For a nonnegative integer $k$, we set 
\[
\D^k(X) := \Lip_{\mathrm b}(X) \times \left( \Lip(X) \right)^k.
\]
Note that $\Lip(X)^k = \Lip(X,\mathbb R^k)$ as sets. 
We regard an element of $\D^k(X)$ as a formal $k$-form on $X$.

We say that a multilinear functional $T : \D^k(X) \to \mathbb R$ is a $k$-{\it current} in $X$ if it satisfies the following conditions: 
\begin{itemize}
\item[(Continuity)] 
Fixing $f \in \Lip_{\mathrm{b}}(X)$, if a sequence $g_i \in \Lip(X, \mathbb R^k)$ converges to $g$ as $i \to \infty$ pointwisely with $\sup_i \Lip(g_i) < \infty$, then we have 
\[
T(f,g_i) \to T(f,g)
\]
as $i \to \infty$; 
\item[(Locality)]
For $(f,g_1, \dots g_k) \in \D^k(X)$, if $g_i$ is constant on $\{f \neq 0\}$ for some $i$, then $T(f,g_1, \dots, g_k)=0$; 
\item[(Finite mass)]
There is a tight finite measure $\mu$ on $X$ such that 
\[
T(f,g_1, \dots, g_k) \le \prod_{i=1}^k \Lip(g_i) \int_X |f|\, d \mu
\]
holds, for every $(f,g_1,\dots g_k) \in \D^k(X)$.
\end{itemize}
Here, when $k=0$, the term $\prod_{i=1}^k \Lip(g_i)$ is regarded as $1$.
The minimal measure satisfying the finite mass axiom for $T$ is denoted by $\|T\|$ and called it the {\it mass measure} of $T$. 
The total measure of $X$ with respect to $\|T\|$ is called the {\it mass} of $T$ and denoted by $\M(T)$.

Let us recall fundamental operations to currents. 
Let $X$ and $Y$ be metric spaces. 
Let $T$ be a $k$-current in $X$ and $\varphi \in \Lip(X,Y)$. 
The pull-back of $(f,g) \in \D^k(Y)$ by $\varphi$ is defined as $\varphi^\sharp (f, g) := (f,g) \circ \varphi$. 
Dually, the push-forward of $T$ by $\varphi$ is defined by $\varphi_\sharp T := T \circ \varphi^\sharp$.
Then, $\varphi_\sharp T$ is a $k$-current in $Y$. 
Since $\Lip_{\mathrm{b}}(X)$ is dense in $L^1(X;\|T\|)$, for a bounded Borel function $f$ and $g \in \Lip(X,\mathbb R^k)$, one can define a value $T(f,g)$ in a canonical way. 
Then, we have a multilinear functional $T : \mathcal B^\infty(X) \times (\Lip(X))^k \to \mathbb R$, where $\mathcal B^\infty(X)$ denotes the space of all bounded Borel functions, which is a unique extension of $T$, having propeties corresponding to continuity, locality and finite mass, of extended forms.
In particular, for a Borel subset $A$ of $X$, we obtain a $k$-current $T \lfloor A$ in $X$ defined by 
\[
T \lfloor A(f,g) := T(\chi_A f, g)
\]
for any $(f,g) \in \D^k(X)$.
Here, $\chi_A$ is the characteristic function of $A$.
The formal exterior differential $d : \D^k(X) \to \D^{k+1}(X)$ is defined by $d(f,g) := (1,f,g)$. 
The {\it boundary} of $T$ is defined by $\partial T := T \circ d$.
When $\partial T$ is a $(k-1)$-current, we call $T$ a {\it normal} current. 
Here, we remark that $\partial T$ is not a current, in general. 
From the locality, $\partial \partial T = 0$ holds.
Therefore, denoting by $\N_k(X)$ the set of all normal $k$-currents in $X$, 
\[
\N_\bullet(X) := \bigoplus_{k=0}^\infty \N_k(X)
\] 
becomes a chain complex with respect to $\partial$.
Further, since the compactness of currents is stable by taking the boundary, the space $\N_\bullet^\cpt(X)$ of all normal currents with compact support is a subcomplex of $\N_\bullet(X)$. 

A typical example of a current is here: 
for a Lebesgure integrable function $\theta$ on $\mathbb R^k$, we set 
\[
\jump{\theta}(f,g) := \int_{\mathbb R^k} \theta f \det (D g) \, d \Hau^k
\]
for $(f,g) \in \D^k(\mathbb R^k)$, 
where $D g$ is the differential of $g \in \Lip(\mathbb R^k, \mathbb R^k)$ in the usual sense, which exists almost everywhere due to Rademacher's theorem. 
Then, $\jump{\theta}$ is a $k$-current in $\mathbb R^k$.

To use later, we recall a characterization of the mass of currents. 
\begin{proposition}[{\cite[Proposition 2.7]{AK}}] \label{prop:mass}
Let $T$ be a $k$-current in a metric space $X$. 
Then, we have 
\[
\M(T) = \sup \sum_{i=1}^\infty T(f_i,g_{i1}, \dots, g_{ik})
\]
where the supremum runs over $(f_i,g_{i1},\dots,g_{ik})$, where $f_i \in \mathcal B^\infty(X)$ and $g_{ij} \in \Lip(X)$ satisfying that $\sum_i |f_i| \le 1$ and $\sup_{i,j} \Lip(g_{ij}) \le 1$.
\end{proposition}

A subset $S$ of $X$ is called countably $\Hau^k$-rectifiable if there exist Borel sets $A_i \subset \mathbb R^k$ and Lipschitz maps $f_i : A_i \to X$ such that 
\[
\Hau^k \left(S \setminus \bigcup_{i=1}^\infty f_i(A_i) \right) = 0
\]
holds. 
We say that a $k$-current in $X$ is {\it rectifiable} if it satisfies the following: 
\begin{itemize}
\item $\|T\|$ is absolutely continuous in $\Hau^k$; 
\item $\|T\|$ is concentrated in a countably $\Hau^k$-rectifiable set.
\end{itemize}
Further, $T$ is {\it integer rectifiable} if it satisfies, in addition, 
\begin{itemize}
\item for any open set $O$ of $X$ and a Lipschitz map $\varphi : X \to \mathbb R^k$, there is an integrable function on $\mathbb R^k$ of integer-valued such that 
\[
\varphi_\sharp (T \lfloor O) = \jump{\theta}
\]
holds. 
\end{itemize}
We say that $T$ is {\it integral} if it is integer rectifiable and is normal.
Let us denote by $\I_k(X)$ the space of all integral $k$-currents in $X$. 

\begin{theorem}[\cite{AK}]
$\I_\bullet(X) := \bigoplus_{k=0}^\infty \I_k(X)$ is a subcomplex of $\N_\bullet(X)$. 
That is, if $T$ is integral, then so is $\partial T$.
\end{theorem}
In particular, $\I_\bullet^\cpt(X) := \N_\bullet^\cpt(X) \cap \I_\bullet(X)$ is also a chain complex.
Note that the homologies of normal currents and integral currents (with compact support) depend on metric structure, in general. 

For a Lipschitz $k$-simplex $f : \triangle^k \to X$ in a metric space $X$, we obtain an integral $k$-current $[f] \in \I_k^\cpt(X)$ with compact support as follows: 
\[
[f] := f_\sharp \jump{\chi_{\triangle^k}}.
\]
The $\mathbb Z$-linear extension of $[\,\cdot\,]$ gives a map 
\[
[\,\cdot\,] : C_k^\Lip(X) \to \I_k^\cpt(X), 
\]
where $C_k^\Lip(X)$ is the subgroup of the usual group $C_k(X)$ of singular chains based on singular Lipschitz simplices. 
Further, $[\, \cdot\,]$ is a chain map, that is, $[\partial f] = \partial [f]$ holds. 
To state the following theorem, we set $C_\bullet^\Lip(X;\mathbb R) = C_\bullet^\Lip(X) \otimes \mathbb R$ which is a free vector space based on singular Lipschitz simplices.

The author proved that the homlogies consisting of suitable classes of currents are topological invariants for nice metric spaces: 
\begin{theorem}[cf. \cite{M:int}] \label{thm:int}
Let $X$ be a metric space which is locally Lipschitz contractible in the sense that it satisfies that for any $x \in X$ and $R > 0$, there is an $r > 0$ such that $B_r(x)$ is Lipschitz contractible in $B_R(x)$.
Then, the inclusion $C_\bullet^\Lip(X) \to C_\bullet(X)$ and the map $[\,\cdot\,]: C_\bullet^\Lip(X) \to \I_\bullet^\cpt(X)$ induce isomorphisms among the induced homologies: 
\[
H_\ast(X;\mathbb Z) \xleftarrow{\cong} H_\ast^\Lip(X;\mathbb Z) \xrightarrow{\cong} H_\ast(\I_\bullet^\cpt(X)).
\]

In addition, if $X$ is an ANR, then the canonical morphisms $C_\bullet(X;\mathbb R) \leftarrow C_\bullet^\Lip(X;\mathbb R) \to \N_\bullet^\cpt(X)$ induce isomorphisms amog the induced homologies. 
\end{theorem}
The fact that $H_\ast(X;\mathbb Z) \leftarrow H_\ast^\Lip(X;\mathbb Z)$ is an isomorphism in the situation of the theorem was also proved in \cite{Y}. 
Alexandrov domains are proved to satisfy the property in the assumption of the theorem (\cite{MY:llc}).
However, for Alexandrov domains, Theorem \ref{thm:int} are proved in several ways. 
For instance, using good coverings (having strong Lipschitz contraction property) in the sense of \cite{MY:good}, one can easily compare the homology of currents with the \v Cech homology of Alexandrov domains. 
In another way, Mongodi (\cite{Mo}) proved that the homologies of currents satisfy the axioms of the ordinary homology theory in suitable coefficients. 
Using it and the fact that Alexandrov domains are ANR, we can also give a proof of Theorem \ref{thm:int} for Alexandrov domains. 

\begin{proof}[Proof of Theorem \ref{thm:current}]
This follows from Theorems \ref{thm:main thm}, \ref{thm:int}, Corollary \ref{cor:Z=R} and the fact that if a metric space $Y$ has the $(n+1)$-dimensional Hausdorff measure zero, then $\N_{n+1}^\cpt(Y) = 0$ (\cite{AK}).
\end{proof}

\subsection{Mass on singular Lipschitz homology}
Gromov considered the mass of elements of the singular homology of Riemannian manifolds (\cite{Gr}). 
Yamaguchi generalized it for a metric space (\cite{Y}). 
The author will define a variant of the mass, to use an advantage of a theory of Kirchheim (\cite{Kir}). 

Let $X$ be a metric space. 
Let $A \subset \mathbb R^k$ be a Borel set and $f : A \to X$ a Lipschitz map. 
For $x \in A$ and $u \in \mathbb R^{k}$, we consider the value
\[
\mathrm{md}(f,x)(u) := \lim_{t \to 0+ \text{ s.t. } x+tu \in A} \frac{d(f(x+tu),f(x))}{t}
\]
whenever it exists. 
Kirchheim (\cite{Kir}) proved that, the functional $\mathrm{md}(f,x) : \mathbb R^k \to \mathbb R$ is actually well-defined and becomes a seminorm, for almost every $x \in A$, which is called 
the metric derivative of $f$ at $x$. 

For a general seminorm $s$ on $\mathbb R^k$, its Jacobian was defiend by 
\[
\mathscr J(s) := (\Hau^{k-1}(\mathbb S^{k-1})) \left( \int_{\mathbb S^{k-1}} s(u)^{-k}\, d \Hau^{k-1}(u) \right)^{-1}.
\]
Here, $\mathbb S^{k-1}$ denotes the unit sphere with respect to the Euclidean metric. 

Using them, we define the {\it mass} of $f$ as follows: 
\[
\mass(f) := \mass(f;A) := \int_A \mathscr J(\mathrm{md}(f,x))\, d \Hau^k(x).
\]
By Kirchheim's coarea formula, we obtain 
\begin{equation} \label{eq:coarea}
\mass(f) = \Hau^k(f(A))
\end{equation}
if $f$ is injective. 
So, it is reasonable to call the value $\mass(f)$ the mass of $f$.

For a singular Lipshitz chain $x = \sum_i a_i f_i \in C_k^\Lip(X;\mathbb R)$, we set 
\[
\mass(x) = \sum_i |a_i| \mass(f_i). 
\]
For an element $\alpha \in H_k^\Lip(X;\mathbb R)$, we set 
\[
\mass(\alpha) := \inf_{z \in \alpha} \mass(z).
\]
Then, $\mass(\,\cdot\,)$ becomes a seminorm on $H_k^\Lip(X;\mathbb R)$.

\begin{lemma} \label{lem:mass-Lip}
Let $X$ be a metric space. 
Then, $[\,\cdot\,] : H_k^\Lip(X;\mathbb R) \to H_k(\N_\bullet^\cpt(X))$ is $\sqrt{k}^k$-Lipschitz with respect to the masses.
Here, we use the convension $0^0=1$.
\end{lemma}
\begin{proof}
It suffices to prove that 
\begin{equation} \label{eq:desired}
\sqrt{k}^k\, \mass(f) \ge \M([f])
\end{equation}
holds for any singular Lipschitz $k$-simplex $f$ in $X$. 
Let $h_i \in \mathcal B^\infty(X)$ and $g_i \in \Lip(X,\mathbb R^k)$ satisfy $\sum_i |h_i| \le 1$ and $\Lip(g_{ij}) \le 1$. 
Then, we have 
\[
\det(D(g_i \circ f,x)) \le 
\Lip(g_i)^k \mathscr J(\mathrm{md}(f,x))
\] 
for almost all $x \in \triangle^k$.
Therefore, we obtain 
\begin{align*}
\sum_i [f](h_i, g_i) 
&\le \sqrt{k}^k\, \sum_i \int_{\triangle^k} h_i \circ f \mathscr J(\mathrm{md}(f,x))\, d \Hau^k(x) \\
&\le \sqrt{k}^k\, \mass(f).
\end{align*}
From this and Proposition \ref{prop:mass}, we have the desired inequality \eqref{eq:desired}.
\end{proof}

The following is an alternative claim of \cite[Theorem 0.1]{Y} which will be proved in Appendix \ref{sec:mass}. 
\begin{theorem}[cf.\! \cite{Y}] \label{thm:mass}
Let $M$ be a closed orientable Alexandrov $n$-space. 
Then, the mass of a fundamental class $[M]_{\Lip} \in H_n^\Lip(M;\mathbb Z)$ is equal to the Hausdorff $n$-measure of $M$: 
\[
\mass([M]_\Lip) = \Hau^n(M).
\]
\end{theorem}

At the end of this section, we give conjectures: 
\begin{conjecture} \upshape
$[\,\cdot\,] : C_\bullet^\Lip(X) \to \I_\bullet^\cpt(X)$ is an actually $1$-Lipschitz map for an arbitrary metric space $X$. 
Further, when $X$ is locally Lipschitz contractible, the induced map $[\,\cdot\,] : H_\ast^\Lip(X;\mathbb Z) \to H_\ast(\I_\bullet^\cpt(X))$ is isometric. 
\end{conjecture}

\begin{conjecture} \upshape
For a singular Lipschitz $k$-simplex $f :\triangle^k \to X$ in a metric space $X$, Yamaguchi (\cite{Y}) gave the mass of it as follows.
\[
\widetilde{\mass}(f) := \sup_{T} \sum_{\tau \in T} \Hau^k(f(\tau)),
\]
where the supremum runs over all triangulation $T$ of $\triangle^k$. 
Then, the two masses in the sense of Yamaguchi and ourselves coincide.
\end{conjecture}
Remark that, by \eqref{eq:coarea}, $\mass(f)=\widetilde{\mass}(f)$ holds when $f$ is injective.

\subsection{Filling radius inequality} \label{subsec:fill rad}

The filling radius was originally defined by Gromov (\cite{Gr}) for Riemannian manifolds. 
Ambrosio and Katz modified it for metric currents as follows. 
\begin{definition}[\cite{AKt}] \upshape \label{def:fill rad}
Let $X$ be a compact metric space and $T$ an integral $n$-current in $X$ with $\partial T=0$.
Then, the infimum of $r >0$ such that 
for any pair $(B,\iota)$ where $B$ is a Banach space and $\iota : X \to B$ is an isometric embedding, there exists $S \in \I_{n+1}^\cpt(B)$ satisfying that $\partial S = \iota_\sharp T$ and $\supp(S)$ is contained in $B_r(\supp(\iota_\sharp T))$, 
is called the {\it filling radius} of $T$ in $X$ and is denoted by $\fillrad{T;X}$.
\end{definition}
Recall that any metric space admits an isometric embedding in a Banach space, for instance, the Kuratowski embedding. 

Ambrosio and Katz proved the following filling radius inequality in terms of metric currents: 
\begin{theorem}[\cite{AKt}] \label{thm:fill rad}
There is an explicit constant $C(n)$ depending only on $n$ such that 
\[
\fillrad{T;X} \le C(n) \M(T)^{1/n}
\]
holds, for any $T \in \I_n(X)$ and any compact metric space $X$. 
\end{theorem}

Let us recall trivial facts. 
Let $X$ be a compact metric space and $\iota : X \to B$ an isometric embedding in a Banach space $B$. 
Taking the cone of $\iota_\sharp T$ (\cite{AK}), we have a trivial upper bound 
\[
\fillrad{T;X} \le 2\, \mathrm{diam}(\mathrm{supp}(T)) \le 2\, \mathrm{diam}(X).
\]

From now on, we consider the situation of the assumption of Theorem \ref{thm:fill rad Alex}. 
Let $M$ be a closed orientable Alexandrov $n$-space. 
Then, we have a fundamental class $[M]_\Lip \in H_n^\Lip(M;\mathbb Z)$ in the singular Lipschitz homology, due to Theorems \ref{thm:main thm} and \ref{thm:int}. 
Passing the isomorphism $[\,\cdot\,] : H_n^\Lip(M;\mathbb Z) \to H_n(\I_\bullet(M))$ given in Theorem \ref{thm:int}, we obtain a non-trivial $n$-current 
\[
\jump{M} := [[M]_\Lip] \in \{ T \in \I_n(M) \mid \partial T = 0 \}.
\]
Here, remark that $\I_{n+1}(M)$ vanishes due to \cite{AK}. 
Then, we set 
\[
\fillrad{M} := \fillrad{\jump{M}; M}
\]
which is called the filling radius of $M$ considered in Theorem \ref{thm:fill rad Alex}.

\begin{proof}[Proof of Theorem \ref{thm:fill rad Alex}]
Let $M$ be as above. 
Applying Theorem \ref{thm:fill rad} to $\jump{M}$, we have 
\[
\fillrad{M} \le C(n) \M(\jump{M})^{1/n}. 
\]
On the other hands, from Lemma \ref{lem:mass-Lip}, we obtain 
\[
\M(\jump{M}) \le \sqrt{n}^{\,n}\, \mass([M]_\Lip).
\]
By the above inequalities and Theorem \ref{thm:mass}, the conclusion of Theorem \ref{thm:fill rad Alex} holds.
\end{proof}

\appendix
\section{The mass of a fundamental class} \label{sec:mass}
In this section, we prove Theorem \ref{thm:mass}, by using a similar argument to the proof of \cite[Theorem 0.1]{Y}. 

For a $k$-simplex $\triangle^k$ and $\alpha \in [0,1]$, we set
\[
\alpha \triangle^k = \{ \alpha(x-x_0) +x_0 \mid x \in \triangle^k \}
\]
where $x_0$ is the barycenter of $\triangle^k$.

\begin{lemma} \label{lem:00}
Let $f : \triangle^k \to \mathbb R^n$ be a Lipschitz $k$-simplex in the Euclidean $n$-space. 
For any $\epsilon > 0$, there is a Lipschitz $k$-simplex $g : \triangle^n \to \mathbb R^n$ such that 
\begin{itemize}
\item $g$ is smooth on $(1/2) \triangle^k$; 
\item $g$ is homotopic to $f$ relative $\partial \triangle^k$; and 
\item $|\mass(g) - \mass(f)| < \epsilon$.
\end{itemize}
\end{lemma}
\begin{proof}
Let us extend $f$ to a Lipschitz map $f : \mathbb R^k \to \mathbb R^n$, where we use the same symbol $f$ denoting the extension of $f$. 
By a standard smoothing argument, we have a family $\{f_\delta\}_{\delta > 0}$ of smooth maps associated with $f$ as follows. 
Let $\rho : \mathbb R^k \to \mathbb R$ be a nonnegative smooth function such that $\int_{\mathbb R^k} \rho\, d \Hau^k = 1$ with compact support.
Then, for $\delta > 0$, we define $f_\delta$ by 
\[
f_\delta(x) := \int_{\mathbb R^k} \delta^{-k} \rho((x-y)/\delta) f(y) \, d \Hau^k(y)
\]
where the integration is $\mathbb R^n$-valued.
Then, we have 
\begin{itemize}
\item $f_\delta$ is $C^\infty$ on $\mathbb R^k$ for all $\delta > 0$;
\item $f_\delta$ converges to $f$ as $\delta \to 0$ uniformly; 
\item $\Lip(f_\delta) \le \Lip(f)$ for all $\delta > 0$; 
\item $D f_\delta(x)$ converges to $Df(x)$ as $\delta \to 0$ for almost every $x \in \mathbb R^k$.
\end{itemize}
All things are verified by direct calculations. 
For a proof of the last assertion, we refer \cite{Honda} for instance. 
By the third and last assertions, we have 
\[
\mass(f_\delta; \triangle^k) \to \mass(f;\triangle^k)
\]
as $\delta \to 0$.
Since $\|f_\delta - f\|_\infty \to 0$, there is $\delta > 0$ such that $|f_\delta(x)-f(x)|<\epsilon$ for all $x \in \mathbb R^k$.
We consider a Lipschitz map $g : \triangle^k \to \mathbb R^n$ defined by, for $x \in \partial \triangle^k$ and $t \in [0,1]$, 
\[
g(tx) := 
\left\{
\begin{aligned}
&f_\delta(2 tx) &&\text{if } t \le 1/2 \\
&(2-2t) f_\delta(x) + (2t-1) f(x) &&\text{if } t \ge 1/2.
\end{aligned}
\right.
\]
Here, the barycenter of $\triangle^k$ is translated to the origin. 
Then, we have 
\[
\Lip(g) \le C(k) \Lip(f) \epsilon
\]
on $\triangle^k - (\frac{1}{2}) \triangle^k$,
where $C(k)$ is a constant depending only on $k$.
Therefore, we obtain 
\begin{align*}
\mass(g; \triangle^k) 
&= \mass(g;(1/2)\triangle^k) + \mass(g; \triangle^k - (1/2) \triangle^k) \\
&\le \mass(f_\delta; \triangle^k) + C(k) \Lip(f)^k \epsilon^k \\
&< \mass(f;\triangle^k) + \{C(k) \Lip(f)^k +1 \} \epsilon^k.
\end{align*}
We also have 
\[
\mass(g;\triangle^k) \ge \mass(f_\delta; \triangle^k) > \mass(f;\triangle^k)-\epsilon.
\]
So, the map $g$ satisfies the conclusion of the lemma. 
Indeed, a homotopy between $f$ and $g$ relative $\partial \triangle^k$ is easily constructed from the concrete form of $g$. 
This completes the proof.
\end{proof}

When the dimensions of the domain and the target of a Lipschitz simplex coincide, we have the following: 
\begin{lemma} \label{lem:01}
Let $f : \triangle^n \to \mathbb R^n$ be a Lipschitz $n$-simplex. 
Then, for any $\epsilon > 0$, there is a Lipschitz chain $g = \sum_{i=1}^N g_i$ such that 
\begin{itemize}
\item each $g_i$ is a Lipschitz $n$-simplex; 
\item $|\mass(f)-\mass(g)| < \epsilon$; 
\item there is $N_1 \le N$ such that $\mass(\sum_{i=N_1+1}^N g_i) < \epsilon$; 
\item $g_i$ are smooth embeddings from $\triangle^n$ into $\mathbb R^n$ for all $i \le N_1$; 
\item $g$ is homologous to $f$.
\end{itemize}
\end{lemma}
\begin{proof}
Due to Lemma \ref{lem:00}, there is a Lipschitz $n$-simplex $g : \triangle^n \to \mathbb R^n$ such that 
$g$ is homotopic to $f$ relative $\partial \triangle^n$, $|\mass(g)-\mass(f)| < \epsilon$ and $g |_{\frac{1}{2} \triangle^n}$ is smooth.
Let us set $S = \{x \in \frac{1}{2} \triangle^n \mid Dg(x) \text{ has rank less than }n \}$.
Then, $S$ is a closed set in $\frac{1}{2} \triangle^n$ and $g(S)$ has zero measure with respect to $\Hau^n$ due to Sard's theorem.
By Vitali's covering theorem, we obtain a countable family $\{\tau_i\}_{i=1}^\infty$ of $n$-simplices contained in $(1/2)\triangle^n - S$ such that $g|_{\tau_i}$ is a smooth embedding for every $i \in \mathbb N$ so that $\{\tau_i\}$ is disjoint and $\Hau^n((1/2) \triangle^n - S) = \sum_{i=1}^\infty \Hau^n(\tau_i)$.
Let us choose a large natural number $N_0$ so that 
\[
\Hau^n \left(\bigcup_{i=N_0+1}^\infty \tau_i \right) < \epsilon.
\]
Let us extend the disjoint family $T_0 = \{\tau_i\}_{i=1}^{N_0}$ to a triangulation $T_1$ of $(1/2)\triangle^n$ and extend $T_1$ to a triangulation $T$ of $\triangle^n$. 
Further, we set $g_\tau := g|_{\tau}$ for all $\tau \in T$.
Then, we have 
\[
\mass \left( \sum_{\tau \in T_1 \setminus T_0} g_\tau \right) = \sum_{i=N_0+1}^\infty \mass(g;\tau_i) \le C(k) (\Lip(f)+1)^k \epsilon
\]
and 
\[
\mass \left( \sum_{\tau \in T \setminus T_1} g_\tau \right) = \mass(g; \triangle^n - (1/2) \triangle^n) \le C'(k) \Lip(f)^k \epsilon^k.
\]
From the construction, the sum $\sum_{\tau \in T} g_\tau$ satisfies the conclusion of the lemma. 
\end{proof}

\begin{proof}[Proof of Theorem \ref{thm:mass}]
We prove $\mass([M]) \le \Hau^n(M)$. 
For every $\epsilon > 0$, we take an $n$-cycle $f = \sum_{i=1}^N a_i f_i \in C_n^\Lip(M;\mathbb Z)$ representing to $[M]$ and 
\[
\mass(f) < \mass([M]) + \epsilon.
\]
By reversing orientations if necessary, we may assume that $a_i \ge 1$ for all $i$.
Further, since $a_i$ are integers, we may assume that $a_i =1$ for all $i$, by allowing overlap of $f_i$'s.
Let $\delta > 0$ such that for any $x \in M_\delta$, there exist a ball $B$ centered at $x$ and a bi-Lipschitz embedding $\varphi : B \to \mathbb R^n$ such that $\Lip(\varphi)$ and $\Lip(\varphi^{-1})$ are at most $1+\epsilon$.
Let us fix $i$. 
Recall that $S_\delta$ is zero measure in $\Hau^n$ (\cite{BGP}, \cite{OS}). 
By Kirchheim's coarea formula (\cite[Theorem 7]{Kir}), we have 
\begin{align*}
\mass(f_i; f_i^{-1}(S_\delta)) \le \int_{S_\delta} \sharp (f_i^{-1}(y)) \, d \Hau^n(y) = 0.
\end{align*}
Hence, 
\[
\mass(f_i) = \mass(f_i; \triangle^n \setminus f_i^{-1}(S_\delta)). 
\]
Since $\triangle^n \setminus f_i^{-1}(S_\delta)$ is open, by using Vitali's covering theorem, 
we can take a triangulation $T^i$ of $\triangle^n$ and a subfamily $T_0^i \subset T^i$ such that 
\begin{itemize}
\item $\left| \mass \left(\sum_{\tau \in T_0^i} f_i |_{\tau}\right) - \mass(f_i) \right| < \epsilon$; 
\item $f_i (\tau)$ is contained in $M_\delta$ for every $\tau \in T_0^i$; 
\item further, for every $\tau \in T_0^i$, $f_i(\tau)$ is contained in a local coordinate neighborhood $B$ in $M_\delta$ with a local chart $\varphi$ so that $\max\{ \Lip(\varphi), \Lip(\varphi^{-1})\} \le 1 +\epsilon$.
\end{itemize}
Let us consider a singular Lipschitz chain 
\[
g := \sum_{i=1}^N \sum_{\tau \in T^i} f_i |_\tau \in C_n^\Lip(M;\mathbb Z). 
\]
By subdividing $T^i$'s if necessary, we may assume that $\partial g = 0$. 
From the construction, $g$ also represents $[M]$ and $\mass(g) = \mass(f)$.

For each $i$ and $\tau \in T_0^i$, by using Lemma \ref{lem:01}, we have an $n$-chain $\tilde f_{i \tau} \in C_n^\Lip(\mathbb R^n;\mathbb Z)$ such that 
\begin{itemize}
\item $\tilde f_{i \tau}$ is homologous to $\varphi \circ f_i |_\tau$; 
\item $|\mass(\tilde f_{i \tau}) - \mass(\varphi \circ f_i |_\tau)| < \epsilon / (\# T_0^i)$; 
\item $\tilde f_{i \tau}$ forms the sum 
\[
\tilde f_{i\tau} = \sum_{\alpha \in S^{i \tau}} \tilde f_{i \tau \alpha}, 
\]
where $S_{i \tau}$ is a finite index set; 
\item there is a subfamily $S_0^{i \tau}$ of $S^{i \tau}$ such that $\tilde f_{i \tau \alpha}$ is a smooth embedding to $\mathbb R^n$ for every $\alpha \in S_0^{i \tau}$ and that 
\[
\sum_{\alpha \in S^{i \tau} \setminus S_0^{i \tau}} \mass(\tilde f_{i \tau \alpha}) < \frac{\epsilon}{ \# T_0^i}.
\]
\end{itemize}

Now, to simplify indicies, using the above construction, we may assume that there is an $n$-cycle $f = \sum_{i=1}^N f_i \in C_n^\Lip(M;\mathbb Z)$ representing $[M]$ such that 
\begin{itemize}
\item there is $N_0 \le N$ such that $f_i$ is a bi-Lipschitz embedding into $M_\delta$ for every $i \le N_0$; 
\item $\mass(f) < \mass([M]) + \epsilon$; 
\item $\sum_{i>N_0} \mass(f_i) < \epsilon/2$.
\end{itemize}
Let us set 
\[
V := S_\delta \cup \bigcup_{i = 1}^{N_0} f_i(\partial \triangle^n) \cup \bigcup_{i>N_0} f_i(\triangle^n). 
\]
Then, $V$ is closed in $M$ and 
\[
\Hau^n(V) \le \sum_{i>N_0} \Hau^n(f_i(\triangle^n)) \le \sum_{i>N_0} \mass(f_i) < \epsilon/2.
\]
For $x \in M \setminus V$ and $i \le N_0$, we set 
\[
\epsilon_i(x) := \left\{ 
\begin{aligned}
&0 && \text{if } f_i^{-1}(x) = \emptyset \\
&1 && \text{if } f_i \text{ preserves orientation} \\ 
&-1 &&\text{if } f_i \text{ reverses orientation}.
\end{aligned}
\right.
\]
Since $f$ represents the fundamental class, we obtain 
\begin{equation} \label{eq:sum}
\sum_{i=1}^{N_0} \epsilon_i(x) = 1 \text{ on } M \setminus V.
\end{equation}
Let us denote by $J_+$ the set of all $i \in \{1, \dots, N_0\}$ such that $f_i$ preserves orientation and set $J_-$ its complement.
\begin{sublemma}[cf.\! \cite{Y}] \label{sublem:sum}
$\sum_{i \in J_-} \mass(f_i) < \epsilon/2$. 
\end{sublemma}
\begin{proof}[Proof of Sublemma \ref{sublem:sum}]
For a subset $\tau \subset M \setminus V$ satisfying that for any $i =1, \dots, N_0$, $\tau \subset f_i(\triangle^n)$ or $\tau \cap f_i(\triangle^n) = \emptyset$, we set 
\begin{align*}
&I_+(\tau) := \{i \le N_0 \mid \epsilon_i(\tau) = 1\}, &I_-(\tau) := \{i \le N_0 \mid \epsilon_i(\tau) = -1\}.
\end{align*}
Further, we set 
\[
E := \{x \in M \setminus V \mid I_-(x) \neq \emptyset\}.
\]

By Vitali's covering theorem, there exist countable families $\{x_j\}$ of points in $E$ and $\{\tau_j\}$ of bi-Lipschitz $n$-simplices with positive orientation in $M \setminus V$ such that 
\begin{itemize}
\item $x_j$ is contained in the interior of $\tau_j$; 
\item $\{\tau_j\}$ is disjoint; 
\item $\Hau^n(E \setminus \bigcup_{j} \tau_j) = 0$; 
\item $\tau_j$ is contained in $\bigcap_{i \in I_+(x_j) \cup I_-(x_j)} f_i(\triangle^n)$.
\end{itemize}
Let $N_1 \gg \sharp J_-$ be a large number. 
We set $\tau = \tau_1 + \dots + \tau_{N_1}$. 
For each $i \le N_0$, $\{f_i^{-1}(\tau_j)\}_{j \le N_1}$ is a disjoint (possibly empty) family of curved $n$-simplicies in $\triangle^n$. 
Let $K_i$ be a triangulation of $\triangle^n \setminus \bigcup_{j \le N_1} f_i^{-1}(\tau_j)$ such that $K_i \cup \{f_i^{-1}(\tau_j)\}_{j \le N_1}$ is a triangulation of $\triangle^n$. 
For $i \le N_0$, we set 
\begin{align*}
\tau^{(i)} := \sum_{\tau_j \subset f_i(\triangle^n)} \tau_j, \text{ and } f_i|_{|K_i|} := \sum_{\sigma \in K_i} f_i|_\sigma. 
\end{align*}
Further, we set
\[
g := \sum_{i=1}^{N_0} f_i|_{|K_i|} + \tau + \sum_{i>N_0} f_i 
\]
which is homologous to $f$ and 
\[
\mass(f) = \mass(g) + 2 \sum_{i \in J_-} \mass(\tau^{(i)}).
\]
In particular, 
\begin{equation*} 
\sum_{i \in J_-} \mass(\tau^{(i)}) < \epsilon /2.
\end{equation*}
Due to the above inequality and $N_1 \gg \sharp J_-$, we have 
\[
\sum_{i \in J_-} \mass(f_i) < \epsilon /2.
\]
This completes the proof. 
\end{proof}
Let us continue the proof of Theorem \ref{thm:mass}. 
\begin{align*}
\Hau^n(M) 
&\ge \Hau^n(M \setminus (V \cup E)) = \int_{M\setminus (V \cup E)} \sum_{i \le N_0} \epsilon_i(x) \, d \Hau^n(x) \\
&= \int_{M\setminus (V \cup E)} \sum_{i \in J_+} \epsilon_i(x) \, d \Hau^n(x) = \sum_{j \in J_+} \int_{M - (V \cup E)} \epsilon_i(x)\, d \Hau^n(x) \\
&= \sum_{i \in J_+} \mass(f_i).
\end{align*}
On the other hands, by Sublemma \ref{sublem:sum}, we obtain 
\begin{align*}
\mass(f) \le \sum_{i \in J_+} \mass(f_i) + \epsilon.
\end{align*}
Therefore, we have the desired inequality 
\[
\mass([M]) \le \Hau^n(M).
\]

Due to the proof of Sublemma \ref{sublem:sum}, we know that $\Hau^n(E) < \epsilon/2$. 
Using this, we obtain 
\[
\Hau^n(M) < \Hau^n(M-(V \cup E)) + \epsilon = \sum_{j \in J_+} \mass(f_i) + \epsilon \le \mass(f) + 2 \epsilon. 
\]
This implies $\Hau^n(M) \le \mass([M])$. 
This completes the proof of Theorem \ref{thm:mass}.
\end{proof}

\section{Geometric interpretation of orientability of Alexandrov spaces} \label{sec:C1-ori}

Let $M$ denote an $n$-dimensional Alexandrov domain which does not meet the boundary. 

Otsu and Shioya defined the notion of $C^1$-atlases on a pair of topological spaces and gave a canonical $C^1$-structure on $M_0 \subset M$ (\cite{OS}). 
Let us recall the definition of $C^1$-atlases on $M_0 \subset M$. 
For an open subset $U \subset M$ and a homeormophism $\varphi : U \to \mathbb R^n$ into an open subset of $\mathbb R^n$, we call $\varphi$ (or a pair $(U,\varphi)$) a local chart of $M$. 
\begin{definition}[\cite{OS}] \upshape 
A family $\Phi = \{(U_\varphi,\varphi)\}_{\varphi}$ of local charts in $M$ is called a $C^1$-{\it atlas on $M_0 \subset M$} if 
\begin{itemize}
\item $\bigcup_{\varphi \in \Phi} U_\varphi \supset M_0$; and 
\item the coordinate transformation $\varphi \circ \psi^{-1}$ is $C^1$ on $\psi(U_\varphi \cap U_\psi \cap M_0)$ for every $\varphi, \psi \in \Phi$. 
\end{itemize}
\end{definition}

Let us recall how to obtain the canonical $C^1$-structure on $M_0 \subset M$. 
For $p \in M_0$ and $\delta > 0$, there exists an $(n,\delta)$-strainer $\{(a_i,b_i)\}_{i=1}^n$ at $p$ (See \cite{BBI}, \cite{BGP} for the definition of strainer). 
We set 
\begin{equation} \label{eq:tilde d}
\tilde d_{a_i}(x) := \frac{1}{\Hau^n(B_{\epsilon}(a_i))} \int_{B_{\epsilon}(a_i)} d(q,x) \, d\Hau^n(q).
\end{equation}
When $\delta \ll 1/n$, $\epsilon, \ell \ll \min_{1 \le i \le n} \{ d(a_i,p), d(b_i,p)\}$, the map $\varphi = (\tilde d_{a_i})_{i=1}^n$ is a local chart with a domain $B_\ell(p)$. 
Otsu and Shioya (\cite{OS}) proved that a family of such local charts is a $C^1$-atlas on $M_0 \subset M$.
A maximal $C^1$-atlas on $M_0 \subset M$ containing this $C^1$-atlas on $M_0 \subset M$ is called the {\it canonical $C^1$-structure on $M_0 \subset M$} and is denoted by $\Phi_M$.

Kuwae, Machigashira and Shioya (\cite{KMS}) implicitely used the following reasonable concept of orientablity, in thier study of $DC^1$-manifolds. 
\begin{definition} \upshape \label{def:C1-ori}
We say that an $n$-dimensional Alexandrov domain $M$ is {\it $C^1$-orientable} if there is a $C^1$-atlas $\Phi$ on $M_0 \subset M$ contained in the canonical $C^1$-structure of $M_0 \subset M$ such that the coordinate transformation $\varphi \circ \psi^{-1}$ has positive Jacobian on $\psi(U_\varphi \cap U_\psi \cap M_0)$ for every $\varphi, \psi \in \Phi$.
Such an atlas $\Phi$ is said to be {\it oriented}. 
\end{definition}

For a smooth manifold, the $C^1$-orientability is equivalent to the usual orientability.
For Alexandrov spaces, we also obtain: 

\begin{theorem} \label{thm:C1-ori}
Let $M$ be an Alexandrov domain which does not meet the boundary. 
Then, $M$ is $C^1$-orientable if and only if it is orientable.
\end{theorem}

To prove this, we recall Kuwae-Machigashira-Shioya's smooth approximation theorem: 
\begin{theorem}[\cite{KMS}] \label{thm:KMS}
Let $\delta > 0$ be small enough with respect to a constant depending on $n$. 
Let $K$ be a compact set contained in $M_\delta$. 
Then, there exists an open neighborhood $U$ of $K$ in $M_\delta$ which admits a smooth atlas in the usual sense which contained in the canonical $C^1$-structure on $M_0 \subset M$.

Further, if $M$ is $C^1$-orientable, then $U$ obtained as above is orientable as a manifold.
\end{theorem}
The last assertion of the theorem was not stated in \cite{KMS}. 
However, the construction of a smooth atlas on $U$ associated with a $C^1$-atlas on $U_0 \subset U_\delta$ given there is easily verified to preserve an orientation. 
Hence, the last statement holds. 
Here, we recall only how to construct $U$ from $K$. 
Since $K$ is compact, it has a uniform lower curvature bound $\kappa$.
Further, the compactness of $K$ implies that, there is a positive number $\ell$ such that each $p \in K$ has an $(n,\delta')$-strainer of length $\ge \ell$, where $\delta'$ is a positive number depending on $\delta$ so that $\lim_{\delta \to 0} \delta' = 0$ and the strainers are associated with $\kappa$-comparison angle.
Let us take $r \ll \ell, \delta$ and a maximal $0.3r$-discrete set $\{p_i\}_{i=1}^N$ of $K$.
Then, $U = \bigcup_{i=1}^N U(p_i, 0.4 r)$ satisfies the desired conclusion.
In particular, if $K$ is connected, then so is $U$.

\begin{lemma}
\label{lem:connected}
Let $V$ be an Alexandrov domain and $\delta \ge 0$. 
Then, $V_\delta$ is connected. 
\end{lemma}
\begin{proof}
Let $n = \dim V$. 
Let $X$ be an Alexandrov $n$-space containing $V$. 
Note that $V_\delta = V \cap X_\delta$. 
Due to Petrunin (\cite{Pet}), we know that $X_\delta$ is convex in the sense that any minimal geodesic joining two points in $X_\delta$ is contained in $X_\delta$. 
In particular, $X_\delta$ itself is connected. 
Let us fix $p, q \in V \cap X_\delta$. 
Since $V$ is connected, there exists a continuous curve $\gamma$ connecting $p$ and $q$ and contained in $V$.
Let us recall that for every point $x$ in $X$, there exists an open neighborhood $W_x$ in $X$ which is convex (\cite{MY:stab}, cf. \cite{Per:Morse}).
Further, such a convex neighborhood can be taken to be arbitrary small. 
Hence, we obtain a covering $\{W_i\}_{i=0}^N$ of $\gamma$ such that each $W_i$ is a convex open set contained in $V$, $W_0 \ni p$, $W_N \ni q$ and $W_{i-1} \cap W_i \neq \emptyset$ for $1 \le i \le N$.
Since $X_\delta$ is dense in $X$, for every $i$, there exists a point $r_i \in W_{i-1} \cap W_i \cap X_\delta$. 
Because $W_i$ are convex, we have a broken geodesic $p r_1 r_2 \cdots r_N q$ contained in $V_\delta$.
This completes the proof.
\end{proof}

\begin{lemma} \label{lem:C1 to ori}
Let $M$ be a $C^1$-orientable Alexandrov $n$-domain. Then, $M_\delta$ is orientable as a manifold for small $\delta > 0$.
\end{lemma}
\begin{proof}
Let $\Phi_+$ be a maximal oriented $C^1$-atlas on $M_0 \subset M$ contained in $\Phi_M$, where $\Phi_M$ is the canonical $C^1$-structure on $M_0 \subset M$.
Since $M$ is second-countable, so is $M_\delta$. 
In particular, $M_\delta$ is $\sigma$-compact, that is, there is a sequence of compact sets $K_\alpha$ in $M_\delta$ so that $\bigcup_{\alpha=1}^\infty K_\alpha = M_\delta$.
Using Lemma \ref{lem:connected}, we may assume that each $K_\alpha$ is path-connected. 
Applying Theorem \ref{thm:KMS} to each $K_\alpha$, we obtain a sequence $\{U_\alpha\}$ of open sets in $M_\delta$ satisfying the following: 
\begin{itemize}
\item $U_\alpha \subset U_{\alpha+1}$ for every $\alpha$ and $\bigcup_{\alpha=1}^\infty U_\alpha = M_\delta$; 
\item each $U_\alpha$ is connected. 
\item there is a smooth atlas $\Phi_\alpha$ on $U_\alpha$ such that $\Phi_\alpha \subset \Phi_+$ for every $\alpha$. 
\end{itemize}
From the last property, $U_\alpha$ is oriented by $\Phi_\alpha$ as a topological manifold. 
Since the orientations of $\Phi_\alpha$ are compatible, $M_\delta$ is orientable. 
\end{proof}

\begin{proof}[Proof of Theorem \ref{thm:C1-ori}]
Suppose that $M$ is $C^1$-orientable. 
By Lemma \ref{lem:C1 to ori}, for small $\delta > 0$, $M_\delta$ is orientable. 
Recall that the proof of Lemma \ref{lem:HS} relies on the condition $\dim (X \setminus X_\top) \le n-2$, where $X$ is as in the lemma. 
On the other hands, since $\delta$ is small and $M$ has no boundary, we have $\dim (M \setminus M_\delta) \le n-2$ (\cite{BGP}, \cite{OS}).
Hence, the proof of Lemma \ref{lem:HS} works for $M$ and $M_\delta$ replaced with $X$ and $X_\top$, respectively. 
This implies that $M$ is orientable. 

Conversely, we suppose that $M$ is orientable and fix an orientation on $M$. 
Then, its open subset $M_\delta$ is oriented. 
Let $\{U_\alpha\}$ be a family of connected open subsets in $M_\delta$ used in the proof of Lemma \ref{lem:C1 to ori}. 
By Theorem \ref{thm:KMS}, we have a smooth oriented atlas $\Phi_\alpha$ on $U_\alpha$ such that $\Phi_\alpha \subset \Phi_M$. 
Then, $\bigcup_{\alpha=1}^\infty \Phi_\alpha$ becomes an oriented $C^1$-atlas on $M_0 \subset M$.
This completes the proof of Theorem \ref{thm:C1-ori}.
\end{proof}

\begin{theorem} \label{thm:n-form}
Let $M$ be an $n$-dimensional Alexandrov domain which does not meet the boundary. 
Then it is orientable if and only if there is a continuous $n$-form $\omega$ on $M_0$ such that $\omega$ is non-zero at each point in $M_0$. 
\end{theorem}
Here, we briefly recall the notion of {\it continuous $n$-forms on $M_0$}. 
The union $\bigsqcup_{p \in M_0} T_p M$ of tangent cones at each regular point becomes a topological vector bundle on $M_0$ of rank $n$, by identifying each $T_p M$ with the Euclidean $n$-space. 
We denote it by $TM$. 
Therefore, we obtain its dual $T^\ast M$ and the exterior products $\bigwedge^k T^\ast M$, which are also topological vector bundles over $M_0$. 
So, the notion of continuity of sections $\omega : M_0 \to \bigwedge^n T^\ast M$ is well-defined. 
\begin{proof}[Proof of Theorem \ref{thm:n-form}]
Suppose that $M$ is orientable. 
By Theorem \ref{thm:C1-ori}, there is an oriented $C^1$-atlas $\Phi$ on $M_0 \subset M$ contained in $\Phi_M$.
Due to \cite{OS}, we obtain a canonical $C^0$-Riemannian metric $g$ associated with $\Phi_M$.
That is, $g$ is a family $g = \{g^\varphi\}_{\varphi \in \Phi_M}$ of matrix-valued functions satisfying that, for $\varphi, \psi \in \Phi_M$, 
\begin{itemize}
\item the domain of $g^\varphi$ is $U_\varphi \cap M_0$, where $U_\varphi$ is the domain of $\varphi$; 
\item $g^\varphi(p)$ is a positive definite symmetric $(n \times n)$-matrix for each $p \in U_\varphi \cap M_0$; 
\item $g^\varphi$ is continuous on $U_\varphi \cap M_0$; 
\item for each $p \in U_\varphi \cap U_\psi \cap M_0$, 
\begin{equation} \label{eq:C0-metric}
g_{ji}^\varphi(p) \frac{\partial x_\varphi^i}{\partial x_\psi^k}(\psi(p)) = g_{ji}^\psi(p) \frac{\partial x_\psi^i}{\partial x_\varphi^k}(\varphi(p))
\end{equation} 
where $x_\chi = (x_\chi^i)$ denotes the coordiante system of $\chi(U_\chi) \subset \mathbb R^n$ for $\chi \in \Phi_M$.
\end{itemize}
Further, Otsu and Shioya proved the compatibility of $g$ and the distance function on $M$.
However, such a property is not needed in this proof.
For $\varphi \in \Phi_M$, we obtain a continuous $n$-form 
\begin{equation} \label{eq:omega-varphi}
\omega^\varphi = \sqrt {\det g^\varphi \circ \varphi^{-1}}\, d x_\varphi^1 \wedge d x_\varphi^2 \wedge \cdots \wedge d x_\varphi^n
\end{equation}
on $\varphi(U_\varphi \cap M_0)$, where $d x_\varphi^i : \varphi(U_\varphi) \to T^\ast \varphi(U_\varphi)$ are the usual exterior derivaties of the coordinate functions $x_\varphi^i$ on $\varphi(U_\varphi) \subset \mathbb R^n$. 
Let us consider the restriction $\{\omega^\varphi\}_{\varphi \in \Phi}$ of the family $\{\omega^\varphi\}_{\varphi \in \Phi_M}$ of local $n$-forms to $\Phi$. 
Since $\Phi$ is oriented, by \eqref{eq:C0-metric} and \eqref{eq:omega-varphi}, it satisfies 
\[
\varphi^\ast \omega^\varphi = \psi^\ast \omega^\psi 
\]
as $n$-forms defined on $U_\varphi \cap U_\psi \cap M_0$ for every $\varphi, \psi \in \Phi$. 
Therefore, we obtain a section $\omega : M_0 \to \bigwedge^n T^\ast M$ defined by $\omega = \varphi^\ast \omega^\varphi$ on $U_\varphi \cap M_0$ for each $\varphi \in \Phi$. 
Then, $\omega$ satisfies the desired condition of the conclusion. 

We prove the converse direction. 
Let us assume that there is a continuous $n$-form $\omega$ on $M_0$ which is pointwisely non-zero.
For each $\varphi \in \Phi_M$, we also obtain a local continuous $n$-form $\omega^\varphi$ defiend on $\varphi(U_\varphi \cap M_0)$ as above. 
By regarding $\bigwedge^n T_p^\ast M$ as $\mathbb R$ for each $p \in M_0$, from the assumption that $\omega(p) \neq 0$, the ``ratio'' 
\[
\omega(p) / \omega^\varphi(\varphi(p))
\]
is regarded as a real number, which varies continuously in $p$.
By Lemma \ref{lem:connected}, 
the sign of the ratio $\omega / \omega^\varphi$ is constant on $M_0 \cap U_\varphi$, whenever $U_\varphi$ is connected. 
Now, we set 
\[
\Phi_+ := \left\{ \varphi \in \Phi_M \left|\,
\begin{aligned} 
&\omega / \omega^\varphi \text{ has a positive sign on } M_0 \cap U_\varphi \\
&\text{and } U_\varphi \text{ is connected} 
\end{aligned} 
\right.
\right\}.
\]
Then, $\Phi_+$ becomes an oriented $C^1$-atlas on $M_0 \subset M$. 
This completes the proof.
\end{proof}


\section{NB-spaces with boundary} \label{sec:NB with boundary} 
We are going to introduce a class $\mathcal N_b$ of ``NB-spaces with boundary'', which is a restricted class to the class of non-branching MCS-spaces with boundary given by Harvey and Searle (\cite{HS}). 
Let us recall the notion of general non-branching MCS-spaces that are defined by a similar way to define general MCS-space (Definition \ref{def:MCS}). 
Spaces consisting of one or two points are called non-branching MCS-spaces of dimension zero. 
An $n$-dimensional MCS-space $X$ is said to be {\it non-branching} if a generator of any point can be taken to be a compact non-branching MCS-space $\Sigma$, where $\Sigma$ is imposed to be connected if $n \ge 2$.
From the definition, one-dimensional non-branching MCS-spaces are topological one-manifolds possibly with boundary.
Following \cite{HS}, we say that $X$ is {\it with boundary} if repeatedly taking generators, finally it becomes a space of single point. 

For a general non-branching MCS-space $X$ of dimension $n$, with $n \ge 1$, we set 
\[
\partial X := \{x \in X \mid H_n(X|x;\mathbb Z_2) \not\cong \mathbb Z_2\}.
\]
The author does not check that our $\partial X$ coincides with the boundary of $X$ 
intended in \cite{HS}, 
because we will soon deal with the restricted class $\mathcal N_b$
and the notion of the boundary of $X$ 
was not given explicitly in \cite{HS}.
However, in this paper, we call $\partial X$ the {\it boundary} of $X$ and $X \setminus \partial X$ the {\it interior} of $X$. 

The reasons why we introduce the class $\mathcal N_b$ is that the class of general non-branching MCS-spaces is too large to describe the underlying topologies of Alexandrov spaces (with boundary).
In general, the boundary of non-branching MCS-space is not an NB-space. 
For instance, if $X = c(S^1 \times [0,1])$, then $\partial X$ is homeomorphic to the cone over the union of two disjoint circles.
Such an $X$ does not admit a metric of Alexandrov spaces. 

Taking into account Theorem \ref{thm:positive boundary}, let us give the definition of the class $\mathcal N_b$: 
\begin{definition} \label{def:NB with boundary} \upshape 
We denote by $\mathcal S^n$ the class of all compact connected NB-spaces, for $n \ge 1$, and by $\mathcal S^0$ the class of discrete spaces consisting of only two points. 
Let $\mathcal S_b^0$ denote the class of all spaces consisting of only one point. 
Let us prepare two classes $\mathcal S_b^n$ and $\mathcal N_b^n$ for $n \ge 1$, inductively. 
\begin{itemize}
\item $\Sigma \in \mathcal S_b^n$ if there is a space $\Lambda \in \mathcal S^{n-1}$ such that $\Sigma$ is homeomorphic to the closed cone $\bar c(\Lambda)$ over $\Lambda$; 
\item $X \in \mathcal N_b^n$ if $X$ is an MCS-space of dimension $n$ such that a generator at each point can be a space belonging to $\mathcal S^{n-1} \cup \mathcal S_b^{n-1}$.
\end{itemize}
Let $\mathcal N_b = \bigcup_{n=1}^\infty \mathcal N_b^n$. 
A space $X$ in $\mathcal N_b$ is called an {\it NB-space with boundary}.
A space which is an NB-space or an NB-space with boundary is called an {\it NB-space possibly with boundary}.
To emphasize, NB-spaces are called an NB-spaces {\it without boundary}.
\end{definition}
Note that if $\Sigma$ is in $\mathcal S^n_b$ for $n \ge 1$, then it is automatically compact and connected. 
Hence, an NB-space with boundary is a non-branching MCS-space. 

From now on, let us denote by $S^0$ and $S^0_+$ are discrete spaces consisting of only two points and of only one point, respectively, that is, $S^0 \in \mathcal S^0$ and $S_+^0 \in \mathcal S^0_b$.
\begin{example} \upshape
If an NB-space with boundary has dimension $n \le 3$, then it is a topological $n$-manifold with boundary.
In four-dimensional case, $c(S^0_+ \ast \mathbb RP^2)$ is an NB-space with boundary, but is not a manifold with boundary.

Our class $\mathcal N_b$ is strictly smaller than the class of general non-branching MCS-spaces with boundary. 
Actually, $c(S^1 \times [0,1])$ is not in $\mathcal N_b^3$ (See also Proposition \ref{prop:bdry is NB}).
\end{example}

The class $\mathcal N_b$ satisfies the following three reasonable statements.
\begin{proposition} \label{prop:interior}
Let $X \in \mathcal N_b$. 
Then, the interior $X \setminus \partial X$ of $X$ is an NB-space.
\end{proposition}

\begin{proposition} \label{prop:collar}
Let $X \in \mathcal N_b$. 
Then, $\partial X$ has a collar neighborhood, that is, there is an open neghborhood $W$ together with a homeomorphism $f : W \to \partial X \times [0,1)$ so that 
$f(x) = (x,0)$ for all $x \in \partial X$.
\end{proposition}

\begin{proposition} \label{prop:bdry is NB}
If $X \in \mathcal N_b$, then $\partial X$ is an NB-space of dimension $\dim X -1$, when $\dim X > 1$.
\end{proposition}

\begin{proof}[Proof of Propositions \ref{prop:interior}--\ref{prop:bdry is NB}]
Let $X \in \mathcal N_b^n$. 
We may assume $n \ge 3$. 
From the defintion, $x \in X \setminus \partial X$ if and only if a generator at $x$ is taken to belong to $\mathcal S^{n-1}$. 
Hence, $X \setminus \partial X$ is an NB-space. 
That is, Proposition \ref{prop:interior} has been proved. 

For $\Lambda \in \mathcal S^{n-2}$, we consider $\Sigma = \bar c(\Lambda)$ which is in $\mathcal S^{n-1}_b$.
Note that $\bar c(\Lambda) \approx S^0_+ \ast \Lambda$, where $A \approx B$ means that $A$ and $B$ are homeomorphic. 
Let $U$ be the cone over $\Sigma$. 
Then, we have 
\[
U \approx [0, \infty) \times c(\Lambda).
\]
By this identification, $\partial U$ corresponds to $\{0\} \times c(\Lambda)$. 
Therefore, $\partial U$ has a collar neighborhood corresponding to $[0,1) \times c(\Lambda)$. 
This observation implies that general $X$ is locally collared. 
Due to Brown (\cite{Bro}), we obtain the conclusion of Proposition \ref{prop:collar}.

Let $\Lambda$, $\Sigma$ and $U = c(\Sigma)$ are the same as above. 
From the definition, we have $(\Sigma, \partial \Sigma) \approx (\bar c(\Lambda), \Lambda \times \{1\})$, where $\bar c(A)$ is defined as $A \times [0,1] / A \times \{0\}$.
Then, we obtain 
\[
(\partial U, o) \approx (c(\partial \Sigma), o),
\]
where $o$ denotes the apex of the cones.
Since $\Lambda \approx \partial \Sigma \in \mathcal S^{n-2}$, $\partial X$ is an $(n-1)$-dimensional NB-space. 
This completes the proof.
\end{proof}

By Proposition \ref{prop:collar} and Theorem \ref{thm:main thm}, we immediately obtain the following, which is a generalization of Corollary \ref{cor:bdry}. 
\begin{corollary} \label{cor:bdry1}
Let $X$ be an $n$-dimensional NB-space with boundary. Then, $H_n(X;G)=0$ for any abelian group. 
\end{corollary}

Let us state several fundamental properties of NB-spaces with boundary.

From the definition, if $X \in \mathcal N_b^n$, then 
\[
\partial X = \{x \in X \mid H_n(X|x; \mathbb Z_2) = 0 \}.
\]

The classes $\mathcal N_b$ and $\mathcal S_b$ are stable under certain geometric constructions, where $\mathcal S_b = \bigcup_{n \ge 0} \mathcal S_b^n$. 
Let us set $\mathcal S = \bigcup_{n \ge 0} \mathcal S^n$ and take $\Sigma \in \mathcal S_b$ and $\Lambda \in \mathcal S \cup \mathcal S_b$. 
Then, $\Sigma \ast \Lambda \in \mathcal S_b$ and $c(\Sigma) \in \mathcal N_b$.
This follows from an observation that $S^0_+ \ast S_+^0 \approx [0,1] \approx S^0 \ast S^0_+$. 
Obviously, if $X \in \mathcal N_b$ and $Y \in \mathcal N_b \cup \mathcal N$, then $X \times Y \in \mathcal N_b$, where $\mathcal N$ denotes the class of all NB-spaces.
For above $\Sigma, \Lambda, X$ and $Y$, we have 
\[
\partial (\Sigma \ast \Lambda) = [(\partial \Sigma) \ast \Lambda] \cup [\Sigma \ast (\partial \Lambda)] \text{ and } 
\partial (X \times Y) = [(\partial X) \times Y] \cup [X \times (\partial Y)]. 
\]
Here, we regard $\Sigma$ and $\Lambda$ as subsets of $\Sigma \ast \Lambda$ in a canonical way, and $\Sigma \ast \emptyset$ as the subset $\Sigma$ in $\Sigma \ast \Lambda$ if $\partial \Lambda = \emptyset$, and use the convention as $\partial S^0 = \emptyset$ and $\partial S^0_+ = S^0_+$.

If $X$ and $Y$ are in $\mathcal N_b$ such that $\partial X$ and $\partial Y$ are homeomorphic by a map $h$, then the space $X \cup_h Y$ obtained by gluing $X$ and $Y$ via $h$ is an NB-space.
Obviously, component-wise gluing also works.
The result of gluing component-wise is an NB-space possibly with boundary.
When $X, Y \in \mathcal S_b$, we have $X \cup_h Y \in \mathcal S$.

When $X \in \mathcal N_b$, the identity $1_{\partial X}$ on $\partial X$ gives an NB-space $X \cup_{1_{\partial X}} X$ as above, which is called the {\it double} of $X$ and is denoted by $D(X)$.

\subsection{NB-spaces with boundary and their orientability}
Due to Proposition \ref{prop:interior}, we define the notion of orientability for NB-spaces with boundary as the following usual way. 
\begin{definition} \upshape
An NB-space with boundary is {\it orientable} if its interior is an orientable NB-space. 
\end{definition}

From now on, $X$ denotes an NB-space with boundary of dimension $n$ and the (co)homologies are taken with coefficients in a principal ideal domain $R$ whenever they are not indicated explicitly. 

\begin{lemma} \label{lem:local ori}
If $X$ is $R$-orientable, then $\partial X$ is locally $R$-orientable.
\end{lemma}
\begin{proof}
We may assume that $n = \dim X \ge 3$.
Let $x \in \partial X$. Then, there exist a neighborhood $U$ of $x$ in $X$ and $\Lambda \in \mathcal S^{n-2}$ such that 
\[
(U, x) \approx (c(\Lambda) \times [0, 1), (o,0)).
\]
Automatically, $\partial U$ corresponds to $c(\Lambda) \times \{0\}$ and it is a conical neighborhood of $x$ in $\partial X$.
From the definition, $U \setminus \partial U$ is $R$-orientable. 
By Corollary \ref{cor:ori product}, $\Lambda$ is $R$-orientable. 
Hence, we conclude that $\partial X$ is locally $R$-orientable.
\end{proof}

\begin{proposition} \label{prop:ori double}
$X$ is $R$-orientable if and only if $D(X)$ is $R$-orientable.
\end{proposition}
\begin{proof}
Since the restriction of an orientation to an open subset is an orientation on the subspace, it suffices to prove that an orientation on $X \setminus \partial X$ induces an orientation on $D(X)$. 
Let us take a copy $X'$ of $X$ and regard $D(X)$ as the union $X \cup X'$ in a natural way. 
In this convention, we have $\partial X = X \cap X' = \partial X'$. 
Let us set $Y = X \setminus \partial X$ and $Y' = X' \setminus \partial X'$. 
From the assumption, there is an orientation $\{o_y\}_{y \in Y}$ on $Y$. 
We define an orientation on $Y'$ by 
\[
o_{\sigma(y)} = \sigma_\ast (o_y) \in H_n(Y'|\sigma(y))
\]
for each $y \in Y$.
For $x \in \partial X$, a conical neighborhood $U \approx c(\Lambda) \times [0,1)$ of $x$ is taken as in the proof of Lemam \ref{lem:local ori}.
Note that $\sigma$ is regarded as an involution on $D(U)$.
By Lemma \ref{lem:local ori}, $D(U) \approx c(\Lambda) \times (-1,1)$ is $R$-orientable. 
We obtain a family $\mathcal O_x$ of local orientations on $D(U)$ induced from the orientation on $c(\Lambda) \times (-1,1)$. 
Then, $\mathcal O_x$ is compatible to the family $\{o_y, o_{\sigma(y)}\}_{y \in Y}$ of local orientations on $Y \cup Y'$. 
Therefore, we obtain an orientation $\bigcup_{x \in \partial X} \mathcal O_x \cup \{o_y, o_{\sigma(y)}\}_{y \in Y}$ on $D(X)$.
This completes the proof.
\end{proof}

\begin{proposition} \label{prop:ori bdry}
If $X$ is $R$-orientable, then $\partial X$ is $R$-orientable. 
\end{proposition}
\begin{proof}
Let $A$ denote a component of $\partial X$. 
By Proposition \ref{prop:ori double}, there is an orientation $\mathcal O = \{o_x\}_{x \in D(X)}$ on $D(X)$. 
For each $x \in A$, let $\Lambda$ and $U$ denote the same thngs as in the proof of Lemma \ref{lem:local ori}.
By using a homeomorphism $(U,\partial U, x) \approx (c(\Lambda) \times [0,1), c(\Lambda) \times \{0\}, (o,0))$, we obtain an isomorphism 
\begin{equation*} \label{eq:double to bdry}
H_n(D(X)|x) \to H_{n-1}(\partial X|x).
\end{equation*}
Thus, we obtain a family of local orientations 
\[
\{o_x \in H_{n-1}(\partial X|x) \}_{x \in \partial X}
\]
on $\partial X$. 
From the construction, this family satisfies the continuity to become an orientation. 
This completes the proof.
\end{proof}

We prove the following Lefschetz-type duality theorem. 
\begin{theorem}
Let $M$ be a compact Alexandrov space with boundary of dimension $n$, where $n \ge 2$.  
If $M$ is orientable, then there is a natual duality map 
\[
D_M : H^k(M, \partial M;\mathbb Q) \to H_{n-k}(M;\mathbb Q)
\]
such that it is isomorphic when $k=n, n-1$.
\end{theorem}
\begin{proof}
	By using Proposition \ref{prop:collar}, we have $H^\ast(M,\partial M) \cong H^\ast_\cpt(M \setminus \partial M)$ and $H_\ast(M) \cong H_\ast(M \setminus \partial M)$.
	Applying Theorem \ref{thm:PD1} to the interior $M \setminus \partial M$, we obtain the conclusion.
\end{proof}

Finally, we give a naive problem.
\begin{problem} \upshape
Find a compact connected NB-space satisfying the conclusion of Theorem \ref{thm:Alex top} which does not admit a metric of an Alexandrov space.
\end{problem}

\end{document}